\documentclass[12pt]{amsart}
\usepackage{amssymb,amscd,amsmath}
\newcommand{\cA}{{\mathcal{A}}}
\newcommand{\cB}{{\mathcal{B}}}

\newcommand{\cG}{{\mathcal{G}}}

\newcommand{\cGD}{{\mathcal{G}}^{*}}

\newcommand{\R}{{\mathbb{R}}}
\newcommand{\F}{{\mathbb{F}}}
\newcommand{\NN}{{\mathbb{N}}}
\newcommand{\Z}{{\mathbb{Z}}}

\newcommand{\AAA}{\mathsf{A}}
\newcommand{\SSS}{\mathsf{S}}

\newcommand{\fp}{{\mathfrak{n}}}
\newcommand{\fps}{{\mathfrak{n}}^{*}}

\newcommand{\al}{\alpha}
\newcommand{\gam}{\gamma}
\newcommand{\e}{\epsilon}
\newcommand{\la}{\lambda}
\newcommand{\varep}{\varepsilon}

\newcommand{\Ind}{{\operatorname{Ind}}}
\newcommand{\Res}{{\operatorname{Res}}}
\newcommand{\Irr}{{\operatorname{Irr}}}
\newcommand{\diag}{{\operatorname{diag}}}
\newcommand{\rk}{{\operatorname{rk}}}

\newcommand{\GL}{{\operatorname{GL}}}
\newcommand{\SL}{{\operatorname{SL}}}
\newcommand{\PSL}{{\operatorname{PSL}}}
\newcommand{\PGL}{{\operatorname{PGL}}}
\newcommand{\GU}{{\operatorname{GU}}}
\newcommand{\PGU}{{\operatorname{PGU}}}
\newcommand{\SU}{{\operatorname{SU}}}
\newcommand{\PSU}{{\operatorname{PSU}}}
\newcommand{\Sp}{{\operatorname{Sp}}}
\newcommand{\PCSp}{{\operatorname{PCSp}}}
\newcommand{\PSp}{{\operatorname{PSp}}}
\newcommand{\CSp}{{\operatorname{CSp}}}
\newcommand{\HS}{{\operatorname{HS}}}
\newcommand{\Spin}{{\operatorname{Spin}}}
\newcommand{\POm}{{\operatorname{P\Omega}}}
\newcommand{\CO}{{\operatorname{CO}}}
\newcommand{\SO}{{\operatorname{SO}}}
\newcommand{\PCO}{{\operatorname{PCO}}}
\newcommand{\PSO}{{\operatorname{PSO}}}
\newcommand{\GO}{{\operatorname{GO}}}
\newcommand{\St}{{\operatorname{St}}}
\newcommand{\rank}{{\operatorname{rank}}}

\newcommand{\FD}{F^*}
\newcommand{\hT}{\hat{T}}
\newcommand{\hs}{\hat{s}}
\newcommand{\bF}{\bar{\F}}

\renewcommand{\bar}{\overline}

\newcommand{\uc}{{\underline c}}
\newcommand{\uh}{{\underline h}}
\newcommand{\uk}{{\underline k}}
\newcommand{\ul}{{\underline l}}

\newcommand{\tw}[1]{{}^#1\!}

\newtheorem{thm}{Theorem}[section]
\newtheorem{cor}[thm]{Corollary}
\newtheorem{prop}[thm]{Proposition}
\newtheorem{exmp}[thm]{Example}
\newtheorem{lem}[thm]{Lemma}
\theoremstyle{remark}

\theoremstyle{definition}

\begin{document}
\title[The Largest Irreducible Representations of Simple Groups]
{The Largest Irreducible Representations\\ of Simple Groups}
\author{Michael Larsen}
\address{Department of Mathematics\\
    Indiana University \\
    Bloomington, IN 47405\\
    U.S.A.}
\email{mjlarsen@indiana.edu}
\author{Gunter Malle}
\address{FB Mathematik, TU Kaiserslautern, 
    67653 Kaiserslautern, Germany}
\email{malle@mathematik.uni-kl.de}
\author{Pham Huu Tiep}
\address{Department of Mathematics\\
    University of Arizona\\
    Tucson, AZ 85721\\
    U. S. A.}
\email{tiep@math.arizona.edu}

\date{Oct. 7, 2010}


\thanks{The authors are grateful to Marty Isaacs for suggesting this problem 
to them. Michael Larsen was partially supported by NSF Grant DMS-0800705, and
Pham Huu Tiep was partially supported by NSF Grant DMS-0901241.}

\begin{abstract}
Answering a question of I. M. Isaacs, we show that the largest degree of 
irreducible complex representations of any finite non-abelian simple group can
be bounded in terms of the smaller degrees. We also study the asymptotic
behavior of this largest degree for finite groups of Lie type. Moreover, we
show that for groups of Lie type, the Steinberg character has largest degree
among all unipotent characters.
\end{abstract}

\maketitle

\section{Introduction}
For any finite group $G$, let $b(G)$ denote the largest degree of any irreducible
complex representation of $G$. Certainly, $b(G)^2 \leq |G|$, and this trivial
bound is best possible in the following sense. One can write
$|G| = b(G)(b(G)+e)$ for some non-negative integer $e$. Then $e = 0$ if and
only if $|G| = 1$. Y. Berkovich showed that $e=1$ precisely when $|G| = 2$
or $G$ is a $2$-transitive Frobenius group, cf. \cite[Theorem 7]{Be}. 
In particular, there is no upper bound on $|G|$ when $e=1$. 
On the other hand, it turns out that $|G|$ \emph{can} be bounded in terms of $e$ if
$e > 1$. Indeed, N. Snyder showed in \cite{Sn} that $|G| \leq ((2e)!)^2$. 

One can ask whether the largest degree $b(G)$ can be bounded in terms of 
the remaining degrees of $G$. More precisely, can one bound 
$$\varep(G) :=
  \dfrac{\sum_{\chi \in \Irr(G),~\chi(1) < b(G)}|\chi(1)|^{2}}{b(G)^{2}}$$
away from $0$ for all finite groups $G$?   
The aforementioned result of Berkovich immediately implies a negative answer to 
this question for general groups.
M. Isaacs raised the question whether there exists a universal constant 
$\varep > 0$ such that $\varep(S) \geq \varep$ for all \emph{simple} groups $S$.
Assuming an affirmative answer to this question, he has 
improved Snyder's bound to the polynomial bound $|G| \leq Be^6$ (for some 
universal constant $B$ and for all finite groups $G$ with $e > 1$),
cf. \cite{I}.

The main goal of this paper is to answer Isaacs' question in the affirmative:

\begin{thm}   \label{main1}
 There exists a universal constant $\varep > 0$ such that 
 $\varep(S) \geq \varep$ for all finite non-abelian simple groups $S$.
\end{thm}

We note that our $\varep$ is implicit because of the proof of Theorem 
\ref{sym}. It would be interesting to get an explicit $\varep$; also,
we do not know of any non-abelian simple group $S$ where $\varep(G) < 1$. As 
pointed out by Isaacs in \cite{I}, if $\varep(S) \geq 1$ for all non-abelian 
simple groups $S$, then his polynomial bound $Be^6$ can be improved to 
$|G| \leq e^6 +e^4$.   
  
For many simple groups $S$, one knows exactly what $b(S)$ is. However, 
for alternating groups $\AAA_n$ there are only asymptotic formulae,
see \cite{VK} and \cite{LS}. For simple classical groups over 
small fields $\F_q$, the right asymptotic for $b(S)$ has not been determined.
In this paper we provide the following lower and upper bounds:

\begin{thm}   \label{main2}
 For any $1 > \varep > 0$, there are some (explicit) constants $A, B > 0$ 
 depending on $\varep$ such that, for any simple algebraic group $\cG$ in 
 characteristic $p$ of rank $n$ and any Frobenius map $F : \cG \to \cG$, the
 largest degree $b(G)$ of the corresponding finite group $G:= \cG^F$ over $\F_q$
 satisfies the following inequalities:  
 $$A(\log_qn)^{(1-\varep)/\gamma} < \frac{b(G)}{|G|_p}
   < B(\log_qn)^{2.54/\gamma}$$
 if $G$ is classical, and 
 $$1 \leq \frac{b(G)}{|G|_p} < B$$
 if $G$ is an exceptional group of Lie type. 
 Here, $\gamma = 1$ if $G$ is untwisted of type $A$, and $\gamma = 2$ otherwise.
\end{thm}

Even more explicit lower and upper bounds for $b(G)$ are proved in \S5 for 
finite classical groups $G$, cf. Theorems \ref{bound4}, \ref{bound5}, and 
\ref{bound6}.

Certainly, any upper bound for $b(G)$ also holds for the largest degree
$b_{\ell}(G)$ of the $\ell$-modular irreducible representations of $G$. Here is 
a lower bound for $b_{\ell}(G)$:
 
\begin{thm}   \label{main3}
 There exists an (explicit) constant $C > 0$ such that, for any simple 
 algebraic group $\cG$ in characteristic $p$, any Frobenius map 
 $F : \cG \to \cG$, and any prime $\ell$, the largest degree $b_{\ell}(G)$
 of $\ell$-modular irreducible representation of $G := \cG^F$ satisfies the
 inequality $b_{\ell}(G)/|G|_p \geq C$.
\end{thm} 

\section{Symmetric Groups}   \label{sec:sym}

We recall some basic combinatorics connected with symmetric groups.
By a \emph{Young diagram}, we mean a finite subset $\Delta$ of
$\Z^{>0}\times \Z^{>0}$ such that for all $(x,y)\in \Z^{>0}\times \Z^{>0}$,
$(x+1,y)\in \Delta$ or $(x,y+1)\in \Delta$ implies $(x,y)\in\Delta$. 
Elements of $\Delta$ are called \emph{nodes}.
We
denote by $Y(n)$ the set of Young diagrams of cardinality $n$.
For any fixed $\Delta$, we let $l$ and $k$ denote the largest $x$-coordinate
and $y$-coordinate in $\Delta$ respectively and define $a_j$ for $1\le j\le k$
and $b_i$ for $1\le i\le l$ by
$$a_j := \max \{i\mid (i,j)\in \Delta\}$$
and likewise
$$b_i := \max \{j\mid (i,j)\in \Delta\}.$$
Thus, for each $\Delta$, we have a pair of mutually transpose partitions
$$n = a_1+\cdots +a_k = b_1+\cdots+b_l.$$
For each $(i,j)\in\Delta$, we define the \emph{hook} $H_{i,j}:=H_{i,j}(\Delta)$
to be the set of $(i',j')\in \Delta$ such that $i'\ge i$, $j'\ge j$, and
equality holds in at least one of these two inequalities. We define the
\emph{hook length}
$$h(i,j) := h_{i,j}(\Delta) := |H_{i,j}(\Delta)| = 1+a_j-i+b_i-j,$$
and set
$$P := P(\Delta) := \prod_{(i,j)\in \Delta} h_{i,j}.$$

Define $\cA(\Delta)$ (resp.  $\cB(\Delta)$) to be the set of nodes that can be added (resp. removed) from $\Delta$ to produce another Young diagram:
$$\cA(\Delta):= \{(i,j)\in \Z^{>0}\times \Z^{>0}\mid \Delta\cup\{(i,j)\}\in Y(n+1)\}$$
and
$$\cB(\Delta):= \{(i,j)\in \Delta\mid \Delta\setminus \{(i,j)\}\in Y(n-1)\}.$$
Thus $\cA(\Delta)$ consists of the pair $(1,k+1)$
and pairs $(a_j+1,j)$ where $j=1$ or $a_j < a_{j-1}$. In particular,
the values $i$ for $(i,j)\in \cA(\Delta)$ are pairwise distinct, so
$$n \ge \sum_{(i,j)\in \cA(\Delta)} (i-1)
    \ge \frac{|\cA(\Delta)|^2 - |\cA(\Delta)|}2,$$
and $|\cA(\Delta)| < \sqrt{2n}+1$. Similarly, $\cB(\Delta)$ consists of
the pairs $(a_j,j)$ where either $j=k$ or $a_j > a_{j+1}$. Hence 
$$n \ge \sum_{(i,j)\in \cB(\Delta)} i 
    \ge \frac{|\cB(\Delta)|^2 + |\cB(\Delta)|}2,$$
and $|\cB(\Delta)| < \sqrt{2n}$. For $(i,j)\in \cA(\Delta)$,
the symmetric difference between $\cA(\Delta)$ and $\cA(\Delta\cup \{(i,j)\})$
consists of at most three elements: $(i,j)$ itself and possibly $(i+1,j)$
and/or $(i,j+1)$. Likewise, the symmetric difference between $\cB(\Delta)$ and 
$\cB(\Delta\setminus \{(i,j)\})$ consists of at most three elements: $(i,j)$
and possibly $(i-1,j)$ and/or $(i,j-1)$.

There are bijective correspondences between elements of $Y(n)$, partitions
$n=\sum_j a_j$, dual partitions $n = \sum_i b_i$, and complex irreducible
characters of $\SSS_n$.  By the hook length formula, the degree of the
character associated to $\Delta$ is $n!/P(\Delta)$.

The branching rule for $\SSS_{n-1}<\SSS_n$ asserts that the restriction to
$\SSS_{n-1}$ of the irreducible representation $\rho(\Delta)$ of $\SSS_n$
associated to $\Delta\in Y(n)$ is the direct sum of $\rho(\Delta\setminus(i,j))$
over all $(i,j)\in \cB(\Delta)$.
By Frobenius reciprocity, it follows that the induction from $\SSS_n$ to
$\SSS_{n+1}$ of $\rho(\Delta)$ is the direct sum of $\rho(\Delta\cup\{(i,j)\})$
over all $(i,j)\in \cA(\Delta)$.
We can now prove the main theorem of this section.

\begin{thm}   \label{sym}
 Let $S\subset \R$ be a finite set. Then there exists $N$ and $\delta>0$
 such that for all $n>N$ and every irreducible character $\phi$ of $\SSS_n$,
 there exists an irreducible character $\psi$ of $\SSS_n$ such that
 $$\frac{\psi(1)}{\phi(1)} \in [\delta,\infty) \setminus S.$$
\end{thm}

\begin{proof}
Equivalently, we prove that for all $\Delta\in Y(n)$ there exists
$\Gamma\in Y(n)$ such that
$$\frac{\dim\rho(\Gamma)}{\dim\rho(\Delta)} 
= \frac{P(\Delta)}{P(\Gamma)} \in [\delta,\infty) \setminus S.$$

Consider the decomposition of
\begin{equation}
  \label{downup}
  \Ind_{\SSS_{n-1}}^{\SSS_n}\Res _{\SSS_{n-1}}^{\SSS_n} \rho(\Delta)
\end{equation}
into irreducible summands of the form $\rho(\Gamma)$. These summands are
indexed by the set $N(\Delta)$ of quadruples $(i_1,j_1,i_2,j_2)$, where
$(i_1,j_1)\in \cB(\Delta)$ and $(i_2,j_2)\in \cA(\Delta\setminus\{(i_1,j_1)\})$.
Clearly, $|N(\Delta)| < \sqrt{2n}(\sqrt{2n}+1)$, while the degree of the
representation (\ref{downup}) equals $n\dim \rho(\Delta)$. Thus, there exists
$\epsilon>0$, depending only on $S$, such that if $n$ is sufficiently large,
either there exists an element of $N(\Delta)$ with corresponding diagram
$\Gamma\in Y(n)$ such that 
$$\frac{\dim\rho(\Gamma)}{\dim\rho(\Delta)} \in [\delta,\infty) \setminus S$$
or there exist at least $\epsilon n$ elements of $N(\Delta)$ with corresponding diagrams $\Gamma$ such that 
\begin{equation}
  \label{inS}
  \frac{\dim\rho(\Gamma)}{\dim\rho(\Delta)} \in S.
\end{equation}
We need only treat the latter case.

Consider octuples $(i_1,j_1,\ldots,i_4,j_4)$ such that $(i_1,j_1,i_2,j_2)$
and $(i_3,j_3,i_4,j_4)$ are in $N(\Delta)$, every $\Gamma$ corresponding to
either of them satisfies (\ref{inS}), the coordinates $i_1,i_2,i_3,i_4$ are
pairwise distinct, and the same is true for the coordinates $j_1,j_2,j_3,j_4$.
The number of such octuples must be at least $\epsilon^2 n^2/2$ if $n$ is
sufficiently large. Let us fix one. We set
$$\Delta_{12} := (\Delta\setminus \{(i_1,j_1)\})\cup\{(i_2,j_2)\}$$
and
$$\Delta_{34} := (\Delta\setminus \{(i_3,j_3)\})\cup\{(i_4,j_4)\}.$$
By the distinctness of the $i$ and $j$ coordinates, we have
$$(i_3,j_3)\in \cB(\Delta_{12})$$
and
$$(i_4,j_4)\in \cA(\Delta_{12}\setminus \{(i_3,j_3)\}).$$
Let 
$$\Delta_{1234} := ((\Delta_{12}\setminus \{(i_3,j_3)\})\cup \{(i_4,j_4)\}.$$

Given $(i,j), (i',j')\in \cA(\Delta)$, we can compare $h_{(i',j')}(\Delta)$ to
$h_{(i',j')}(\Delta\cup \{(i,j)\})$. If $i\neq i'$ and $j\neq j'$, the hook
lengths are equal, but if $i=i'$ or $j=j'$, then 
$$h_{(i',j')}(\Delta\cup \{(i,j)\}) = h_{(i',j')}(\Delta)+1.$$
From this formula, we deduce that
$$\frac{P(\Delta)P(\Delta_{1234})}{P(\Delta_{12})P(\Delta_{34})} 
= \frac{a(a+2)}{(a+1)^2}\cdot\frac{b(b-2)}{(b-1)^2},$$
where
$$a = h_{(\min(i_2,i_4),\min(j_2,j_4))}(\Delta),\ 
  b = h_{(\min(i_1,i_3),\min(j_1,j_3))}(\Delta).$$
Letting $S^2 = \{s_1 s_2\mid s_1,s_2\in S\}$, we conclude that
$$\frac{P(\Delta_{1234})}{P(\Delta)} \in 
  \left(\frac{a(a+2)}{(a+1)^2}\cdot
  \frac{b(b-2)}{(b-1)^2}\right)S^2.$$
As long as $\delta$ is chosen less than $(9/16)(\min S)^2$, this value is
automatically greater than $\delta$.  It remains to show that we can choose
the octuple $(i_1,\ldots,j_4)$ such that the value is not in $S$.  

There are finitely many values of $t$ such that $tS^2 \cap S$ is non-empty,
and we need only consider values of $a$ and $b$ for which
$$\frac{a(a+2)}{(a+1)^2}\cdot\frac{b(b-2)}{(b-1)^2}$$
lies in this finite set.  We claim that for each value $t$, the set of octuples
which achieves this value is $o(n^2)$. The claim implies the theorem. To
prove the claim we note that there are $O(n^{3/2})$ possibilities for
$(i_1,j_1,i_2,j_2,i_3,j_3)$.  Given one such value, $b$ is determined, so if
$t$ is fixed, so is $a$.  For a given value of $a$ and given $i_2$ and $j_2$,
there are at most two possibilities for $(i_4,j_4)\in \cA(\Delta)$ with
$h_{(\min(i_2,i_4),\min(j_2,j_4))}(\Delta)$ achieving this fixed value.
The claim follows.
\end{proof}

\begin{cor}   \label{alt}
 There is some constant $\varep > 0$ such that $\varep(\AAA_n) \geq \varep$
 for all $n \geq 5$.
\end{cor}

\begin{proof}
Choose $S := \{2,1,1/2\}$ and apply Theorem \ref{sym}. 
Let $\chi \in \Irr(\AAA_n)$ be of degree $b := b(\AAA_n)$ and let
$\phi \in \Irr(\SSS_n)$ be lying above $\chi$;
in particular $\phi(1) = rb$ with $r = 1$ or $2$. By Theorem \ref{sym}, there 
is some $\psi \in \Irr(\SSS_n)$ such that $\psi(1)/\phi(1) \geq \delta$ and 
$\psi(1)/\phi(1) \notin S$. Now let $\rho \in \Irr(\AAA_n)$ be lying under
$\psi$; in particular, $\rho(1) = \psi(1)/s$ with $s = 1$ or $2$. Then
$\rho(1)/\chi(1) = (\psi(1)/\phi(1)) \cdot (r/s)$, and so
$\rho(1)/\chi(1) \geq \delta/2$ and $\rho(1)/\chi(1) \neq 1$. It follows
that $\varep(\AAA_n) \geq \delta^2/4$ for $n\ge N$.
\end{proof}

\section{Comparing Unipotent Character Degrees of Simple Groups of Lie Type}

Each finite simple group $S$ of Lie type, say in characteristic $p$, has the
\emph{Steinberg character} $\St$, which is irreducible of degree $|S|_p$. We 
refer the reader to \cite{C} and \cite{DM} for this, as well as basic facts
on Deligne-Lusztig theory. The main aim of this section is the proof of the
following comparison result:

\begin{thm}   \label{thm:steinberg}
 Let $\cG$ be a simple algebraic group in characteristic $p$,
 $F : \cG \to \cG$ a Frobenius map, and $G = \cG^F$ be the corresponding 
 finite group of Lie type. Then the degree of the Steinberg character of $G$
 is strictly larger than the degree of any other unipotent character.
\end{thm}

By the results of Lusztig, unipotent characters of isogenous groups have
the same degrees, so it is immaterial here whether we speak of groups of
adjoint or of simply connected type; moreover, all unipotent characters
have the center in their kernel, so they can all be considered as characters
of the corresponding simple group.

It is easily checked from the formulas in \cite[\S13]{C} and the data in
\cite{Lu} that Theorem~\ref{thm:steinberg} does in fact
hold for exceptional groups of Lie type. The six series of classical groups
are handled in Corollaries~\ref{cor:StGL} and~\ref{cor:StGU} and
Proposition~\ref{prop:Stclass} after some combinatorial preparations. On the
way we derive some further interesting relations between unipotent character
degrees.

\subsection{Type $\GL_n$}
For $q>1$ and $\uc=(c_1<\ldots<c_s)$ a strictly increasing sequence we set
$$[\uc]:=\prod_{i=1}^s(q^{c_i}-1)$$
and $\underline c+m:=(c_1+m<\ldots<c_s+m)$ for an integer $m$.

\begin{lem}   \label{lem:ineq}
 Let $q\ge2$, $s\ge1$.
 \begin{enumerate}
  \item[\rm(i)] $\frac{q^a-1}{q^{a-1}-1}\le\frac{q^b-1}{q^{b-1}-1}$ if and
   only if $a\ge b$.
  \item[\rm(ii)] Let $\uc=(c_1<\ldots<c_s)$ be a strictly increasing
   sequence of integers, with $c_1\ge2$. Then:
   $$q^s< \frac{[\uc]}{[\uc-1]}<q^{s+1}.$$
 \end{enumerate}
\end{lem}

\begin{proof}
The first part is obvious, and then the second follows by a $2s$-fold
application of~(i) since
$$q^s< \frac{q^{c_s}-1}{q^{c_s-s}-1}
    =\prod_{i=1}^s\frac{q^{c_s-s+i}-1}{q^{c_s-s+i-1}-1}
  \le\prod_{i=1}^s\frac{q^{c_i}-1}{q^{c_i-1}-1}
  \le \prod_{i=1}^s\frac{q^{i+1}-1}{q^i-1}=\frac{q^{s+1}-1}{q-1}<q^{s+1}.
$$
\end{proof}

We denote by $\chi_\la$ the unipotent character of $\GL_n(q)$ parametrized by
the partition $\la$ of $n$. Its degree is given by the quantized hook formula
$$\chi_\la(1)
  =q^{a(\la)}\frac{(q-1)\cdots(q^n-1)}{\prod_{h}(q^{l(h)}-1)},$$
where $h$ runs over the hooks of $\la=(a_1\ge\ldots\ge a_r)$, and
$a(\la)=\sum_{i=1}^r(i-1)a_i$ (see for example \cite[(21)]{Ol} or \cite{Ma1}).

\begin{prop}   \label{prop:compGL}
 Let $\la=(a_1\ge\ldots\ge a_{r-1}>0)\vdash n-1$ be a partition of $n-1$ and
 $\mu,\nu$ the partitions of $n$ obtained by adding a node at $(r,1)$, $(i,j)$
 respectively, where $i<r$ and $a_i=j-1$. Then for all $q\ge2$ the
 corresponding unipotent character degrees of $\GL_n(q)$ satisfy
 $$q^{-j-1}\chi_{\mu}(1)<\chi_{\nu}(1)<q^{2-j}\chi_{\mu}(1)\le\chi_{\mu}(1).$$
\end{prop}

\begin{proof}
According to the hook formula, we have to consider the hooks in
$\mu,\nu$ of different lengths. These lie in the $1$st column, 
the $i$th rows and in the $j$th column. Let $\uh=(1<h_2<\ldots<h_r)$ denote
the hook lengths in the $1$st column, $\uk=(k_1<\ldots<k_{j-1})$ the hook
lengths in the $i$th row and $\ul=(l_1<\ldots<l_{i-1})$ the hook lengths in the
$j$th column of $\mu$. Write
$\uh'=(h_2<\ldots<h_r)$ and $\uk'=(0<k_1<\ldots<k_{j-1})$. Then a threefold
application of Lemma~\ref{lem:ineq}(ii) shows that
$$\begin{aligned}
  \chi_{\nu}(1)
  &=q^{a(\nu)-a(\mu)}\frac{[\uh]}{[\uh'-1]}
   \frac{[\uk]}{[\uk'+1]}\frac{[\ul]}{[\ul+1]}\chi_{\mu}(1)\\
   &< q^{-r+i}q^rq^{1-j}q^{1-i}\chi_{\mu}(1)
   =q^{2-j}\chi_{\mu}(1)
   \le \chi_{\mu}(1),
\end{aligned}$$
since $j\ge2$. The other inequality is then also immediate.
\end{proof}

Note that $\nu$ is the partition obtained from $\mu$ by moving one
node from the last row (which contains a single node) to some row higher up.
Since clearly any partition of $n$ can be reached by a finite number of
such operations from $(1)^n$, we conclude:

\begin{cor}   \label{cor:StGL}
 Any unipotent character of $\GL_n(q)$ other than the Steinberg character
 $\St$ has smaller degree than $\St$.
\end{cor}

A better result can be obtained when $q\ge3$, since then the upper bound in
Lemma~\ref{lem:ineq}(ii) can be improved to $q^{s+1/2}$. In that case, `moving
up' any node in a partition leads to a smaller unipotent degree:

\begin{prop}   \label{prop:dominance}
 Let $q\ge3$, and $\nu\ne\mu$ two partitions of $n$ with $\nu\rhd\mu$ in the
 dominance order. Then the corresponding unipotent character degrees of
 $\GL_n(q)$ satisfy $\chi_{\nu}(1)<\chi_{\mu}(1)$.
\end{prop}

\begin{proof}
In our situation, $\nu$ can be reached from $\mu$ by a sequence of steps of
moving up a node in a partition. Consider one such step, where the node at
position $(r,s)$ is moved to position $(i,j)$, with $j>s$. A similar estimate
as in the proof of Proposition~\ref{prop:compGL}, but with the improved
upper bound from Lemma~\ref{lem:ineq}, leads to the result.
\end{proof}

\begin{exmp}
The previous result fails for $q=2$; the smallest counterexample occurs for
$n=6$, $\mu=(2)^3\lhd\nu=(3)(2)(1)$, where $\chi_\mu(1)=5952<\chi_\nu(1)=6480$.
\end{exmp}

\subsection{Type $\GU_n$}
The analogue of Proposition~\ref{prop:compGL} is no longer true for the
unipotent characters of unitary groups, in general. Still, we can obtain a
characterization of the Steinberg character by comparing with character
degrees in $\GL_n(q)$.

\begin{prop}
 Any partition $\la$ of $n$ has $r=\lceil n/2\rceil$ distinct hooks
 $h_1,\ldots,h_r$ of odd lengths $l(h_i)\le 2i-1$, $1\le i\le r$.
\end{prop}

\begin{proof}
We proceed by induction on $n$. The result is clear for 2-cores, i.e.,
triangular partitions. Now let $\la=(a_1\le\ldots\le a_r)$ be a partition
of~$n$ which is not a 2-core, with corresponding $\beta$-set
$B=\{a_1,a_2+1,\ldots,a_r+r-1\}$. The hook lengths of $\la$ are just the
differences $j-i$ with $j\in B$, $i\notin B$, $i<j$ (see \cite[Lemma~2]{Ol}).
Since $\la$ is not a 2-core, there exists $j\in B$ with $j-2\notin B$. Let
$B'=\{j-2\}\cup B\setminus\{j\}$, the $\beta$-set of a partition $\mu$ of
$n-2$. We now compare hook lengths in $B'$ and in $B$: hooks in $B'$ from
$k>j$, $k\in B'$, to $j$ become hooks from $k$ to $j-2$ in $B$, and
hooks from $j-2$ to $k\notin B'$, $k<j-2$, become hooks from $j$ to
$k$ in $B$. In both cases, the length has increased by~2. But we have one
further new hook in $B$: either from $j$ to $j-1$ (if $j-1\notin B'$), or
from $j-1$ to $j-2$ (if $j-1\in B'$), of length~1. So indeed, in both cases
we've produced hooks of the required odd lengths in $\la$.
\end{proof}

\begin{prop}   \label{prop:GLvsGU}
 Let $\la$ be a partition of $n$. Then the degree of the unipotent
 character of $\GL_n(q)$ indexed by $\la$ is at least as big as the
 corresponding one of $\GU_n(q)$.
\end{prop}

\begin{proof}
It's well-known that the degree of the unipotent character of $\GU_n(q)$
indexed by $\la$ is obtained from the one for $\GL_n(q)$ by formally
replacing $q$ by$-q$ in the hook formula above and adjusting the sign. Now
let $h_1,\ldots,h_r$ denote the sequence of hooks of odd length from the
previous result. Observe that the numerators in the hook formula for $\GL_n(q)$
and $\GU_n(q)$ differ by the factor
$\prod_{i=1}^r (q^{2i-1}+1)/(q^{2i-1}-1)$.
Since $(q^a+1)/(q^b+1)<(q^a-1)/(q^b-1)$ when $b<a$, the claim now follows from
the hook formula.
\end{proof}

It seems that the only case with equality, apart from the trivial cases
$1$ and $\St$, occurs for the partition $(2)^2$ of~$4$.

Since the degree of the Steinberg character of $\GL_n(q)$ and $\GU_n(q)$ is
the same, the following is immediate from Corollary~\ref{cor:StGL} and
Proposition~\ref{prop:GLvsGU}:

\begin{cor}   \label{cor:StGU}
 Any unipotent character of $\GU_n(q)$ other than the Steinberg character
 $\St$ has smaller degree than $\St$.
\end{cor}

\subsection{Other classical types}
The unipotent characters of the remaining classical groups $G = G(q)$
(i.e. symplectic and orthogonal groups) are labelled by \emph{symbols},
whose definition and basic combinatorics we now recall (we refer to \cite{Ma1}
for the version of the hook formula given here). A \emph{symbol} $S=(X,Y)$
is a pair of strictly increasing sequences $X=(x_1<\ldots<x_r)$,
$Y=(y_1<\ldots<y_s)$ of non-negative integers. The \emph{rank} of $S$ is then
$$\rk(S)=\sum_{i=1}^r x_i+\sum_{j=1}^s y_j
  -\left\lfloor\left(\frac{r+s-1}{2}\right)^2\right\rfloor.$$
The symbol $S'=(\{0\}\cup (X+1),\{0\}\cup (Y+1))$ is said to be
\emph{equivalent} to $S$, and so is the symbol $(Y,X)$. The rank is constant
on equivalence classes. The \emph{defect} of $S$ is
$d(S)=||X|-|Y||$, which clearly is also invariant under equivalence. \par
Lusztig has shown that the unipotent characters of classical groups of rank~$n$
are naturally parametrized by equivalence classes of symbols of rank~$n$, with
those of odd defect parametrizing characters in type $B_n$ and $C_n$, those
of defect $\equiv0\pmod4$
characters in type $D_n$, and those of defect $\equiv2\pmod4$ characters in
type $\tw2D_n$. (Here, each so-called degenerate symbol, where $X=Y$,
parametrizes two unipotent characters in type $D_n$.) \par
The degrees of unipotent characters are most conveniently given by an
analogue of the hook formula for $\GL_n(q)$, as follows. A \emph{hook of $S$}
is a pair $(b,c)\in\NN_0^2$ with $b<c$ and either $b\notin X$, $c\in X$,
or $b\notin Y$, $c\in Y$. Thus, a hook of $S$ is nothing else but a hook
(as considered in Section~\ref{sec:sym} and for type $A$ above) of the
permutation with associated $\beta$-set either $X$ or $Y$. A \emph{cohook
of $S$} is a pair $(b,c)\in\NN_0^2$ with $b\le c$ and either $b\notin Y$,
$c\in X$, or $b\notin X$, $c\in Y$. (Note that the possibility $b=c$ for
cohooks was excluded in \cite{Ol} which led to a less smooth hook formula than
in \cite{Ma1}.) We also set
$$a(S):=\sum_{\{b,c\}\subseteq S}\min\{b,c\}-\sum_{i\ge1}\binom{r+s-2i}{2},$$
where the sum runs over all 2-element subsets of the multiset $X\cup Y$ of
entries of $S$.
The degree of the unipotent character $\chi_S$ of a finite classical group
$G=G(q)$ parametrized by $S$ is then given as
$$\chi_S(1)=q^{a(S)}\frac{|G|_{q'}}{\prod_{(b,c)\text{ hook}}(q^{c-b}-1)
            \prod_{(b,c)\text{ cohook}}(q^{c-b}+1)},$$
where the products run over hooks, respectively cohooks of $S$ (see
\cite[Bem.~3.12 and~6.8]{Ma1}). It can be
checked that this is constant on equivalence classes. It is also clear from
this that the unipotent characters in types $B_n$ and $C_n$ have the same
degrees.

\begin{lem}   \label{lem:ineq2}
 Let $q\ge2$, $s\ge1$.
 \begin{enumerate}
  \item[\rm(i)] $\frac{q^a+1}{q^{a-1}+1}\le\frac{q^b+1}{q^{b-1}+1}$ if and
   only if $a\le b$.
  \item[\rm(ii)] Let $(c_1<\ldots<c_s)$ be a strictly increasing
   sequence of integers, with $c_1>0$. Then:
   $$q^{s-1}\le q^s/2< \prod_{i=1}^s \frac{q^{c_i}+1}{q^{c_i-1}+1}<q^s.$$
 \end{enumerate}
\end{lem}

The proof is immediate and entirely similar to the one of Lemma~\ref{lem:ineq}.

\begin{prop}   \label{prop:Stclass}
 Let $S$ be a symbol of rank~$n$, parametrizing a unipotent character of
 $G=G(q)$ of rank~$n$. Then $\chi_S(1)\le\St(1)$, where $\St$ denotes the
 Steinberg character of $G$, with equality only if $\chi_S=\St$.
\end{prop}

\begin{proof}
We'll describe an algorithm changing the entries of a given symbol, preserving
its rank and the parity of its defect, but increasing the corresponding
character degree, which eventually leads to the symbol parametrizing the
Steinberg character. Note that the Steinberg characters of $D_n(q)$ and
$\tw2D_n(q)$ have the same degree, so that we may switch freely between
symbols of any even defect. \par
First assume that $S=((0,1,\ldots,x),())$ is a so-called cuspidal symbol.
Then the symbol $S'=((0,\ldots,x-1),(x))$ has same rank, same $a$-value and
same parity of defect, but larger degree unless $x=0$, in which case $S$
parametrizes the trivial character of $B_0$. So from now on we may suppose
that $S$ contains at least one 'hole', that's to say, not both sequences
$X,Y$ of $S$ are complete intervals starting at~0.

First replace $S$ by an equivalent symbols such that $0\in X\cap Y$ and
$1\notin X\cap Y$. We may and will then assume that $1\notin X$. Now let
$b$ be maximal such that $S$ has a hook $(b,b+1)$. (Such an $b$ exists
since $S$ is not cuspidal.)
Let $x=|\{i\in X\mid i\le b\}|$, $y=|\{i\in Y\mid i\le b\}|$ and $m=x+y$
(the number of entries in $S$ below $b+1$). By the definition of $b$, there
are no holes in $S$ above $b$.  \par
First assume that $m\le 2b-2$ and let $S'$ be the symbol obtained from $S$ as
follows: first replace $0$ by $1$ in $X$, then replace $b+1$ by $b$ in $X$
if $(b,b+1)$ was a hook of $X$, respectively in $Y$ if it was a hook of $Y$.
Clearly the new symbol has the same rank and the same defect. We have
$a(S')=a(S)+m-1$, and the quotient of the contributions by the hooks and
cohooks can be seen to be larger than $q^{2b-2m-1}$ by Lemmas~\ref{lem:ineq}
and~\ref{lem:ineq2}, so that $\chi_{S'}(1)>q^{2b-m-2}\chi_S(1)\ge \chi_S(1)$
by our assumption on~$m$.
Thus any unipotent character degree is smaller than one for a symbol $S$ with
$m\ge 2b-1$, so there are at most three 'holes' in the entries of $S$. But
in fact, if $m=2b-1$ the above process leads to the better inequality
$\chi_{S'}(1)>q^{2b-m-1}\chi_S(1)\ge \chi_S(1)$, whence $S$ has at most two
holes. \par
Now note that if $S=(X,Y)$, with $x=\max X>\max Y+1$ and $x-1\in X$, then
$S'=(X\setminus\{x\},Y\cup\{x\})$ has at least as many holes as $S$ and
larger associated character degree, since
$\prod_{i=2}^k(q^i-1)/(q^i+1)>(q-1)/(q+1)$, so we may assume that $\max X$,
$\max Y$ differ by at most~1. \par
If $S$ has just one hole, up to equivalence it is of the form
$S=((1,\ldots,x),(0,\ldots,y))$ with $y\in\{x+1,x,x-1\}$. But such a symbol
parametrizes the Steinberg character of $G$ of type $D_n$ when $y=x-1$,
of $B_n$ (and $C_n$) when $y=x$, respectively of $\tw2D_n$ when $y=x+1$.
\par
Next assume that $S$ has two holes. If there is $b\le\max(X\cup Y)$ such that
$b\notin X\cup Y$, then we may pass to an equivalent symbol
$S=((1,\ldots,x),(1,\ldots,y))$, where without loss $y\in\{x,x-1\}$ and so
$n=x+y$. If $n$ is odd, so $x=(n+1)/2$, $y=(n-1)/2$, we have
$$\chi_S(1)=q^{a(S)}\prod_{i=1}^y\frac{q^{2(x+i)}-1}{q^{2i}-1}
  \le q^{\binom{n}{2}}\prod_{i=1}^y q^{2x+1}=q^{n^2-1}<q^{n^2}=\St(1),$$
while for $n$ even, $x=y=n/2$, we find
$$\chi_S(1)=\frac{q^{a(S)}}{q^n+1}
   \prod_{i=1}^y\frac{q^{2(x+i)}-1}{q^{2i}-1}
  \le \frac{q^{\binom{n}{2}}}{q^n+1}\prod_{i=1}^y q^{2x+1}
  =\frac{q^{n^2}}{q^n+1}<q^{n^2-n}=\St(1).$$
\par
So finally, $S$ has two holes with different values, in which case there is an
equivalent symbol $S$ obtained from $S_0:=((0,\ldots,x),(1,\ldots,y))$ by
removing at most one entry, different from~0, and $x\in\{y,y\pm1\}$. If
$x>y$ and $S,S_0$ differ in the second row, then the symbol obtained from $S$
by moving $x$ from the first to the second row has larger degree. If $x=y$ and
the two holes lie in different rows, moving one of them to the other row
generates a larger degree. In the last two remaining cases (with $x>y$ and
$|X|\le|Y|$, respectively $x=y$ and $|X|=|Y|-2$) the above procedure of
reducing the smallest and increasing the largest hole leads to a symbol with
larger degree.
\end{proof}

This completes the proof of Theorem~\ref{thm:steinberg}.

\section{Theorem \ref{main1} for Simple Groups of Lie Type}
Notice that to prove Theorem \ref{main1} we can ignore any finite number of 
non-abelian simple groups, in particular the $26$ sporadic groups (of course
one can find out the exact value of $\varep(S)$ for each of them; in particular,
one can check using \cite{Atlas} that $\varep(S)>1$ for all the sporadic
groups). Thus, in view of Corollary~\ref{alt}, it remains to prove
Theorem~\ref{main1} for simple groups of Lie type.   

We begin with some estimates:

\begin{lem}   \label{sum}
 Let $q \geq 2$. Then the following inequalities hold.
 \begin{enumerate}
  \item[\rm (i)] $\prod^{\infty}_{i=1}(1-1/q^i) > 1-1/q -1/q^2 +1/q^5 \geq \exp(-\al/q)$,
where $\al = 2\ln(32/9) \approx 2.537$.
  \item[\rm (ii)] $\prod^{\infty}_{i=2}(1-1/q^i) > 9/16$.
  \item[\rm (iii)] $\prod^{\infty}_{i=k}(1+1/q^i)$ is smaller than $2.4$
   if $k = 1$, $1.6$ if $k = 2$, $1.28$ if $k = 3$, and $16/15$ if $k = 5$.
  \item[\rm (iv)] $1 < \prod^{n}_{i=1}(1-(-1/q)^{i}) \leq 3/2$.
 \end{enumerate}
\end{lem}

\begin{proof}
(i) As mentioned in \cite{FG} (see the paragraph after Lemma 3.4 of \cite{FG}), 
a convenient way to prove these estimates is to use Euler's pentagonal
number theorem \cite[p. 11]{A}:
\begin{equation}\label{euler}
  \begin{array}{ll}\prod^{\infty}_{i=1}(1-\frac{1}{q^i}) & 
  = 1 + \sum^{\infty}_{n=1}(-1)^n(q^{-n(3n-1)/2} + q^{-n(3n+1)/2})\\
  & = 1 - q^{-1} - q^{-2} + q^{-5} + q^{-7} - q^{-12} - q^{-15} + \cdots
  \end{array}
\end{equation}
Since $q^{-m} \geq \sum^{\infty}_{i=m+1}q^{-i}$, finite partial sums of this
series yield arbitrarily accurate upper and lower bounds for
$\prod^{\infty}_{i=1}(1-q^{-i})$. In particular, truncating the series
(\ref{euler}) at the term $q^{-5}$ yields the first inequality. Next, consider
the function
$$f(x) := 1-x-x^2+x^5 - \exp(-\al x)$$ 
for the chosen $\al$. The choice of $\al$ ensures that $f(1/2) = 0 = f(0)$,
and $f''(x) < 0$ for all $x\in [0,1/2]$. It follows that $f(x) \geq 0$ on
$[0,1/2]$, yielding the second inequality.    

(ii) Clearly, 
$$\prod^{\infty}_{i=2}(1-\frac{1}{q^{i}}) \geq \prod^{\infty}_{i=2}(1-\frac{1}{2^{i}})
  = 2\prod^{\infty}_{i=1}(1-\frac{1}{2^{i}}) > 
  2(1-\frac{1}{2}-\frac{1}{4}+\frac{1}{32}) = \frac{9}{16}.$$

(iii) Applying (\ref{euler}) with $q = 4$ and truncating the series at the term
$q^{-7}$ we get $\prod^{\infty}_{i=1}(1-4^{-i}) < 0.6876$. Applying (\ref{euler}) 
with $q = 2$ and truncating the series at the term $q^{-15}$ we get 
$\prod^{\infty}_{i=1}(1-2^{-i}) > 0.2887$. Now 
$$\prod^{\infty}_{i=1}(1+\frac{1}{q^{i}}) \leq \prod^{\infty}_{i=1}(1+\frac{1}{2^{i}})
  = \frac{\prod^{\infty}_{i=1}(1-\frac{1}{4^{i}})}
  {\prod^{\infty}_{i=1}(1-\frac{1}{2^{i}})} < \frac{0.6876}{0.2887} < 2.382.$$
The other bounds can be obtained by using this bound and noting that
$$\prod^{\infty}_{i=k}(1+\frac{1}{q^{i}}) \leq \prod^{\infty}_{i=k}(1+\frac{1}{2^{i}})
  = \frac{\prod^{\infty}_{i=1}(1+\frac{1}{2^{i}})}
  {\prod^{k-1}_{i=1}(1+\frac{1}{2^{i}})}.$$  

(iv) follows from the estimates 
$$(1-\frac{1}{q^{2k}})(1+\frac{1}{q^{2k+1}}) < 1 < 
   (1+\frac{1}{q^{2k-1}})(1-\frac{1}{q^{2k}})$$ 
for any $k \geq 1$. 
\end{proof}

First we prove Theorem~\ref{main1} for the exceptional groups of Lie type:

\begin{prop}   \label{exc}
 There is a constant $\varep > 0$ such that $\varep(S) \geq \varep$ for 
 all simple exceptional groups $S$ of Lie type.
\end{prop}

\begin{proof}
It suffices to show that there is some constant $\delta > 0$ such that 
every simple exceptional group $S$ of Lie type, in characteristic say $p$, 
has an irreducible character $\rho$ such that $\delta b(S) \leq \rho(1) < b(S)$. 
In this proof, we will use the notation $^aX_r(q^a)$ to indicate the type of
$S$, where $a = 1$ for untwisted groups and $a = 2,3$ for twisted groups
(so $q$ can be irrational for the Suzuki and 
the Ree groups, and $r \leq 8$ is the rank of the underlying algebraic group). 
Let $H$ be the group of adjoint type corresponding to $S$; 
in particular, $S = [H,H]$, and $|H/S| \leq 3$. 
The list of character degrees of $H$ is given on Frank L\"ubeck's 
website \cite{Lu}. A detailed inspection of this list would yield an explicit 
$\delta$ (and hence $\varep$), but we will give a short, implicit proof.
One can see that there is a finite list of monic polynomials 
$f_i(q)$, $1 \leq i \leq n_0$ (which are divisors of cyclotomic 
polynomials $\Phi_m(q^a)$ with $m \leq 30$) such that 
every character degree $f(q)$ of $H$ is $c(f)$ times a product of some
powers of these $f_i(q)$, where $c(f)$ is the leading coefficient for $f$. 
It follows that there is an absolute constant $\beta > 0$ such that 
$f(q) \leq \beta q^{\deg(f)}$ for all such $f$ and all $q$. If $|H|_p = q^N$,
then $\deg(f) \leq N$ (as one can easily check). Hence we have shown that 
$b(H) \leq \beta q^N$.

First assume that $b(S)$ is attained at some character $\chi$ of $S$ of 
degree $\neq q^N = \St(1)$. Then we can set $\rho := \St$ and observe that 
$$1 > \rho(1)/\chi(1) = q^N/b(S) \geq q^N/b(H) \geq 1/\beta.$$ 
Now assume that $b(S) = q^N$. Then the finite group $L$ of 
simply connected type corresponding to
$S$ contains a regular semisimple element $t$. Consider the semisimple character
$\chi_t$ of $H$ labeled by the conjugacy class of $t$ (recall that 
$H$ is adjoint, so the centralizer of $t$ in the underlying algebraic 
group is connected). Then 
$C_L(t)$ is a maximal torus of $L$, and so $|C_L(t)| \leq (q+1)^r$. Since 
$\chi_t(1) = |H|_{p'}/|C_L(t)|$ and $r \leq 8$, we see that there is a universal
constant $\gam > 0$ such that $\chi_t(1) > \gam q^N$ for all $q$. Now let 
$\rho \in \Irr(S)$ lie below $\chi_t$. Since $\chi_t(1)$ is coprime to $p$,
$\rho(1) \neq q^N = b(S)$. Finally,
$\rho(1) \geq \chi_t(1)/|H/S| \geq \chi_t(1)/3$, whence 
$\rho(1)/b(S) \geq \gam/3$.             
\end{proof}

The same proof as above establishes Theorem \ref{main1} for simple classical
groups of \emph{bounded} rank. To handle simple classical groups of arbitrary
rank, we will need the following result of G. M. Seitz \cite[Theorem 2.1]{S}:

\begin{thm}[Seitz]   \label{bound1}
 Let $\cG$ be a simple, simply connected algebraic group over the algebraic
 closure of a finite field of characteristic~$p$,
 $F~:~\cG \to \cG$ a Frobenius endomorphism of $\cG$, and let 
 $L := \cG^F$. Then $b(L) \leq |L|_{p'}/|T_0|$, where $T_0$ is a maximal torus
 of $L$ of minimal order. For $q$ sufficiently large, this is in fact an
 equality (namely, whenever that torus contains at least one regular element).
\end{thm}

In what follows, we will view our simple classical group $S$ as $L/Z(L)$,
where $L = \cG^F$ as in Theorem \ref{bound1}. We also consider the pair
$(\cGD,\FD)$ dual to $(\cG,F)$ and the group $H := (\cGD)^{\FD}$ dual to $L$.

\begin{thm}   \label{clas}
 Let $S$ be a finite simple classical group. Suppose that $S$ is not 
 isomorphic to any of the following groups:
 $$\left\{\begin{array}{l}
  \SL_n(2),~\Sp_{2n}(2),~\Omega^{\pm}_{2n}(2),\\
  \PSL_n(3) \mbox{ with }5 \leq n \leq 14,~
  \PSU_n(2) \mbox{ with }7 \leq n \leq 14,\\
  \PSp_{2n}(3) \mbox{ or } \Omega_{2n+1}(3) \mbox{ with }4 \leq n \leq 17,
  ~\POm^{\pm}_{2n}(3) \mbox{ with }4 \leq n \leq 30,\\
  \POm^{\pm}_{8}(7), ~\POm^{\pm}_{2n}(5) \mbox{ with }4 \leq n \leq 6.
  \end{array}\right.$$
 Then $\varep(S) > 1$.
\end{thm}

\begin{proof}
1) First we consider the case $S = \PSL_n(q)$ with $q \geq 3$. Then
$\cG = \SL_n(\bF_q)$, $\cGD = \PGL_n(\bF_q)$,  
$L = \SL_n(q)$, $H = \PGL_n(q)$, and the maximal tori of minimal order in Theorem
\ref{bound1} are the maximally split ones, of order $(q-1)^{n-1}$. Hence
$$b(S) \leq b(L) \leq B/(q-1)^{n-1}, \mbox{ where }
  B:= |L|_{p'} = (q^2-1)(q^3-1) \cdots (q^n-1).$$
Now we consider a maximal torus $T$ of order $q^{n-1}-1$ in $H$, with full 
inverse image $\hT = C_{q^{n-1}-1} \times C_{q-1}$ in $\GL_n(q)$. We will 
show that, in the generic case, the regular semisimple elements in $T$ will
produce enough irreducible characters of $S$, all of degree less than
$b(S)$, and with the sum of squares of their degrees exceeding $b(S)^2$.

Assume $n \geq 4$.
A typical element $\hs$ of $\hT \cap L$ is $\GL_n(\bF_q)$-conjugate to
$$\diag\left(\al,\al^q, \ldots ,\al^{q^{n-2}},\al^{(1-q^{n-1})/(q-1)}\right),$$
where $\al \in \F_{q^{n-1}}^{\times}$. Let 
$$X := \left(\cup^{n-2}_{i=1}\F_{q^i}^{\times} \cup 
  \{x \in \bF_q^{\times} \mid x^{n(q-1)} = 1 \}\right) \cap \F_{q^{n-1}}^{\times}.$$
Also set $m := \lfloor (n-1)/2 \rfloor$. Then for $n \geq 6$ we have 
$n-m \geq 4$, and so  
$$|X| < \sum^{m}_{i=0}q^i +n(q-1) \leq \frac{q^{m+1}-1}{2} + n(q-1) 
 \leq \frac{q^{n-3}-1}{2} + n(q-1) < \frac{q^{n-1}-1}{2}$$ 
since $q \geq 3$. Direct calculations show that $|X| < (q^{n-1}-1)/2$ also
for $n = 4,5$. Thus there are at least $(q^{n-1}-1)/2$ elements $\al$ 
in $\F_{q^{n-1}}^{\times}$ that do not belong to $X$. Consider $\hs$ for any such 
$\al$. Then all the $n$ eigenvalues of $\hs$ are distinct, and exactly one 
of them (namely $\beta:=\al^{-(q^{n-1}-1)/(q-1)}$)
belongs to $\F_q$. Suppose that $x \in \GL_n(\bF_q)$ centralizes $\hs$
modulo $Z(\GL_n(\bF_q))$: $x\hs x^{-1} = \gam \hs$. Comparing the determinant,
we see that $\gam^n = 1$. Suppose that for some $i$ with $0 \leq i \leq n-2$, 
$\gam\beta = \al^{q^i}$. Then $(\al^{q^i})^{n(q-1)} = (\gam\beta)^{n(q-1)} = 1$,
and so $\al \in X$, a contradiction. Hence $\gam\beta = \beta$, i.e.
$\gam = 1$ and $x$ centralizes $\hs$, and clearly 
$C_{\GL_n(\bF_q)}(\hs)$ is a maximal torus. So if 
$s \in T$ is the image of $\hs$, then 
$C_{\cGD}(s) = C_{\GL_n(\bF_q)}(\hs)/Z(\GL_n(\bF_q))$ is 
connected and a maximal torus of $\cGD$; in particular,
$s$ is regular. Also, $s \in \PSL_n(q) = [H,H]$.
Hence each such $s$ defines an irreducible character $\chi_s$ of $L$, of 
degree $B/|T|$, which is trivial at $Z(L)$. So we can view $\chi_s$ as
an irreducible character of $S$. Each such $s$ has at most $q-1$ inverse images
$\hs \in \hT \cap L$. Moreover, since $|N_{H}(T)/T| = n-1$, the $H$-conjugacy 
class of $s$ intersects $T$ at $n-1$ elements. We have therefore produced
at least $(q^{n-1}-1)/2(q-1)(n-1)$ irreducible characters $\chi_s$ of $S$, each of 
degree 
$$\chi_s(1) = B/|T| = (q^2-1)(q^{3}-1) \ldots (q^{n-2}-1)(q^n-1).$$
Note that $\chi_s(1) < q^{2+3+ \ldots +(n-2)+n} = \St(1) \leq b(S)$. Hence,
to show that $\varep(S) > 1$, it suffices to verify that
$$\frac{q^{n-1}-1}{2(q-1)(n-1)} \cdot \left(\frac{B}{q^{n-1}-1}\right)^2 >
  \left(\frac{B}{(q-1)^{n-1}}\right)^2,$$
equivalently, $(q-1)^{2n-3} > 2(n-1)(q^{n-1}-1)$. The latter inequality 
holds if $q = 3$ and $n \geq 15$, or if $q = 4$ and $n \geq 5$, or if 
$q \geq 5$ and $n \geq 4$. It is straightforward to check that 
$\varep(S) > 1$ when $n = 2,3$ or $(n,q) = (4,4)$ (using \cite{GAP} for the
last case).    
       
\medskip
2) Next let $S = \PSU_n(q)$ with $n \geq 3$. Then
$\cG = \SL_n(\bF_q)$, $\cGD = \PGL_n(\bF_q)$,  
$L = \SU_n(q)$, $H = \PGU_n(q)$. The maximal tori of minimal order in Theorem
\ref{bound1} have order at least $(q^2-1)^{n/2}/(q+1)$. Hence
$$b(S) \leq b(L) \leq B(q+1)/(q^2-1)^{n/2}, \mbox{ where }
  B:= |L|_{p'} = \prod^{n}_{i=2}(q^i-(-1)^i).$$
Now we consider a maximal torus $T$ of order $q^{n-1}-(-1)^{n-1}$ in $H$, with full 
inverse image $\hT = C_{q^{n-1}-(-1)^{n-1}} \times C_{q+1}$ in $\GU_n(q)$. We will
follow the same approach as in the case of $\PSL_n(q)$.

Assume that $n \geq 4$, and moreover $q \geq 3$ if $4 \leq n \leq 7$.
A typical element $\hs$ of $\hT \cap L$ is $\GL_n(\bF_q)$-conjugate to
$$\diag\left(\al,\al^{-q}, \ldots ,\al^{(-q)^{n-2}},\al^{((-q)^{n-1}-1)/(q+1)}\right),$$
where $\al \in C_{q^{n-1}-(-1)^{n-1}} < \bF_{q}^{\times}$. 
Let $Y$ be the set of elements in $C_{q^{n-1}-(-1)^{n-1}}$ that belong to
a cyclic subgroup $C_{q^k-(-1)^k}$ of $\bF_{q}^{\times}$ for some $1 \leq k < n-1$
or have order dividing $n(q+1)$. Assume $n \geq 9$ and set 
$m:= \lfloor (n-1)/2 \rfloor$. Then $n-m \geq 5$ and so  
$$\begin{array}{ll}|Y| & \leq \sum^{m}_{i=1}(q^i-(-1)^i)+n(q+1) 
  \leq \sum^{m}_{i=0}q^i+n(q+1)  \\ \\
  & = \dfrac{q^{m+1}-1}{q-1}+n(q+1)
  \leq \dfrac{q^{n-3}-1}{2} + n(q+1) < \dfrac{q^{n-1}-1}{2}.\end{array}$$
Direct calculations show that $|Y| < (q^{n-1}-1)/2$ also for 
$4 \leq n \leq 8$ (recall that we are assuming $q \geq 3$ when $4 \leq n \leq 7$).
Thus there are at least $(q^{n-1}-(-1)^{n-1})/2$ elements of $C_{q^{n-1}-(-1)^{n-1}}$
that do not belong to $Y$. Consider $\hs$ for any such $\al$. 
Then all the $n$ eigenvalues
of $\hs$ are distinct, and exactly one of them (namely $\al^{((-q)^{n-1}-1)/(q+1)}$)
belongs to $C_{q+1} < \bF_q^{\times}$. Arguing as in the $\PSL
$-case, we see that 
if $x \in \GL_n(\bF_q)$ centralizes $\hs$ modulo $Z(\GL_n(\bF_q))$, then $x$ 
actually centralizes $\hs$, and $C_{\GL_n(\bF_q)}(\hs)$ is a maximal torus. So if 
$s \in T$ is the image of $\hs$, then 
$C_{\cGD}(s) = C_{\GL_n(\bF_q)}(\hs)/Z(\GL_n(\bF_q))$ is 
connected and a maximal torus of $\cGD$; in particular,
$s$ is regular. Also, $s \in \PSU_n(q) = [H,H]$.
Hence each such $s$ defines an irreducible character $\chi_s$ of $L$, of 
degree $B/|T|$, which is trivial at $Z(L)$. So we can view $\chi_s$ as
an irreducible character of $S$. Each such $s$ has at most $q+1$ inverse images
$\hs \in \hT \cap L$. Moreover, since $|N_{H}(T)/T| = n-1$, the $H$-conjugacy 
class of $s$ intersects $T$ at $n-1$ elements. We have therefore produced
at least $(q^{n-1}-(-1)^{n-1})/2(q+1)(n-1)$ irreducible characters $\chi_s$ of $S$, each of 
degree 
$$\chi_s(1) = \frac{B}{|T|} = 
\frac{\prod^{n}_{i=2}(q^{i}-(-1)^i)}{q^{n-1}-(-1)^{n-1}}.$$
Note that $\chi_s(1) < q^{2+3+ \ldots +(n-2)+n} = \St(1) \leq b(S)$. Hence,
to show that $\varep(S) > 1$, it suffices to verify that
$$\frac{q^{n-1}-(-1)^{n-1}}{2(q+1)(n-1)} \cdot \left(\frac{B}{(q^{n-1}-(-1))^{n-1}}
  \right)^2 >
  \left(\frac{B(q+1)}{(q^2-1)^{n/2}}\right)^2,$$
equivalently, $(q^2-1)^n > 2(n-1)(q^{n-1}-(-1)^{n-1})(q+1)^3$. The latter 
inequality holds if $q = 2$ and $n \geq 15$, or if $q = 3$ and $n \geq 6$, or if 
$q \geq 4$ and $n \geq 4$. It is straightforward to check that 
$\varep(S) > 1$ when $n = 3$ (and $q \geq 3$), or $(n,q) = (4,2)$, $(5,2)$, 
$(6,2)$, $(4,3)$, $(5,3)$
(using \cite{Lu} in the last case).    

\medskip
3) Here we consider the case $S = \PSp_{2n}(q)$ or $\Omega_{2n+1}(q)$ with 
$n \geq 2$ and $q \geq 3$ (and $q$ is odd in the $\Omega$-case). Then 
$L = \Sp_{2n}(q)$, resp. $L = \Spin_{2n+1}(q)$. 
The maximal tori of minimal order 
in Theorem \ref{bound1} have order at least $(q-1)^n$. Hence
$$b(S) \leq b(L) \leq B/(q-1)^n, \mbox{ where }
  B:= |L|_{p'} = \prod^{n}_{i=1}(q^{2i}-1).$$  
To simplify the computation, we will view $S$ as a normal subgroup of index 
$\leq \kappa := \gcd(2,q-1)$ of the Lie-type group of adjoint type 
$K := \PCSp_{2n}(q)$, resp. $K := \SO_{2n+1}(q)$. Then any semisimple element 
in the dual group $K^* = \Spin_{2n+1}(q)$, resp. $\Sp_{2n}(q)$, has 
connected centralizer (in the underlying algebraic group).  
Now we consider a maximal torus $T$ of order $q^n-1$ in $K^*$, and let 
$X$ be the set of elements in $T$ of order dividing $q^k \pm 1$ 
for some $k$ with $1 \leq k \leq n-1$. Setting $m := \lfloor n/2 \rfloor$ 
we have
$$|X| \leq \sum^{m}_{i=1}((q^i+1)+(q^i-1)) < 2\frac{q^{m+1}-1}{q-1}
  \leq q^{n-1}-1 < \frac{q^n-1}{3},$$
if $n \geq 3$. One can also check by direct computation that 
$|X| \leq (q^n-1)/2$ if $n = 2$. Hence there are
at least $(q^n-1)/2$ elements of $T$ that are regular semisimple.
Each such $s$ defines an irreducible character $\chi_s$ of $K$ of 
degree $B/|T|$. Moreover, since $|N_{K^*}(T)/T| = 2n$, the $K^*$-conjugacy 
class of $s$ intersects $T$ at $2n$ elements. We have therefore produced
at least $(q^n-1)/4n$ irreducible characters $\chi_s$ of $K$, each of 
degree 
$$\chi_s(1) = \frac{B}{|T|} = 
  \frac{\prod^{n}_{i=1}(q^{2i}-1)}{q^n-1} < q^{n^2} = \St(1) \leq b(S).$$
First we consider the characters $\chi_s$ which split over $S$. 
They exist only when $|K/S| = \kappa = 2$. Then the irreducible 
constituents of their restrictions to 
$S$ are all distinct, and the sum of squares of the 
degrees of the irreducible components of each $(\chi_s)|_S$ is $\chi_s(1)^2/\kappa$.
On the other hand, among the $\chi_s$ which are irreducible over $S$, at most 
$\kappa$ of them can restrict to the same (given) irreducible character of $S$.   
Hence, to show that $\varep(S) > 1$, it suffices to verify that
$$\frac{1}{\kappa} \cdot \frac{q^n-1}{4n} \cdot \left(\frac{B}{q^n-1}\right)^2 >
  \left(\frac{B}{(q-1)^n}\right)^2,$$
equivalently, $(q-1)^{2n} > 4\kappa n(q^n-1)$. The latter 
inequality holds if $q = 3$ and $n \geq 18$, or if $q = 4$ and $n \geq 4$, or if 
$q = 5$ and $n \geq 3$, or if $q \geq 7$ and $n \geq 2$. Using \cite{Atlas} and
\cite{GAP} one can check that $\varep(S) > 1$ when $(n,q) = (2,3)$, 
$(2,4)$, $(2,5)$, $(3,3)$, $(3,4)$.    
       
\medskip
4) Finally, we consider the cases $S = \POm^{\pm}_{2n}(q)$, where 
$n \geq 4$ and $q \geq 3$.  We set $\epsilon$ to $1$ or $-1$ 
in the split and non-split cases respectively.
Then $L = \Spin^{\pm}_{2n}(q)$, and 
the maximal tori of minimal order 
in Theorem \ref{bound1} have order at least $(q-1)^n$. Hence
$$b(S) \leq b(L) \leq B/(q-1)^n, \mbox{ where }
  B:= |L|_{p'} = (q^n-\e) \cdot \prod^{n-1}_{i=1}(q^{2i}-1).$$  
As in 3), we will view $S$ as a normal subgroup of index 
$\leq \kappa := \gcd(4,q^n-\e)$ of the Lie-type group of adjoint type 
$H := P(\CO^{\pm}_{2n}(q)^\circ)$. Then any semisimple element in the dual group
$L$ has connected centralizer.  
Now we consider a maximal torus $T$ of order $q^n-\e$ in $L$, and let 
$Y$ be the set of elements in $T$ of order dividing $q^k + 1$ or $q^k-1$
for some $k$ with $1 \leq k \leq n-1$. As in 3) we see that
$|X| \leq (q^n-\e)/3$ since $n \geq 4$. Hence there are
at least $2(q^n-\e)/3$ elements of $T$ that are regular semisimple.
Each such $s$ defines an irreducible character $\chi_s$ of $H$ of 
degree $B/|T|$. Moreover, since $|N_{L}(T)/T| = 2n$, the $L$-conjugacy 
class of $s$ intersects $T$ at $2n$ elements. We have therefore produced
at least $(q^n-\e)/3n$ irreducible characters $\chi_s$ of $H$, each of 
degree 
$$\chi_s(1) = B/|T| = 
  \prod^{n-1}_{i=1}(q^{2i}-1) < q^{n(n-1)} = \St(1) \leq b(S).$$
The restriction $(\chi_s)|_S$ contains an irreducible constituent 
$\rho_s$ of degree at least $\chi_s(1)/\kappa$. Conversely, each 
$\rho \in \Irr(S)$ can lie under at most $\kappa$ distinct irreducible
characters of $H$. 
Hence, to show that $\varep(S) > 1$, it suffices to verify that
$$\frac{1}{\kappa} \cdot \frac{q^n-\e}{3n} \cdot 
  \left(\frac{B}{\kappa(q^n-\e1)}\right)^2 >
  \left(\frac{B}{(q-1)^n}\right)^2,$$
equivalently, $(q-1)^{2n} > 3\kappa^{3}n(q^n-\e)$. The latter 
inequality holds unless $q = 3$ and $n \leq 30$, 
or if $q = 5$ and $n \leq 6$, or if $q = 7$ and $n = 4$. 
\end{proof}

Now we handle the remaining infinite families of simple classical groups
over~$\F_2$.

\begin{thm}   \label{clas2}
 If $S$ is any of the following simple classical groups over $\F_2$: 
 $$\SL_n(2), ~~\Sp_{2n}(2)', ~~\Omega^{\e}_{2n}(2),$$ 
 then one of the following statements holds:

 {\rm (i)} There exists $\psi \in \Irr(S)$ with 
 $81/512 \leq \psi(1)/b(S) < 1$;

 {\rm (ii)} $\varep(S) > 9/16$.
\end{thm}

\begin{proof}
The ``small'' groups $\SL_3(2)$ and $\Sp_4(2)' \cong \AAA_6$ are easily handled
using \cite{Atlas}. Also set $q = 2$.

1) First we consider the case $S = \SL_n(2)$ with $n \geq 4$. 
Then the dual group is isomorphic to
$S$. In particular, any character $\chi \in \Irr(S)$ of largest degree $b(S)$ 
can be parametrized by $((s),\phi)$, where $(s)$ is the conjugacy class of 
a semisimple element $s \in S$ and $\phi$ is a unipotent character of 
the centralizer $C := C_S(s)$. Such a centralizer is isomorphic to
$$\GL_{k_1}(2^{d_1}) \times \ldots \times \GL_{k_r}(2^{d_r}),$$
where $k_i, d_i \geq 1$, $k_1d_1 \geq k_2d_2 \geq \ldots \geq k_rd_r$, and
$\sum^{r}_{i=1}k_id_i = n$. Moreover, for each $d$, the number of indices 
$i$ such that $d_i = d$ is at most the number of conjugacy classes of 
semisimple elements in $\GL_{kd}(2)$ with centralizer $\cong \GL_{k}(2^d)$, i.e.
the number of monic irreducible polynomials $f(t)$ of degree $d$ over $\F_2$. 
Since $\chi(1) = (S:C)_{2'} \cdot \psi(1)$ and
$\chi(1) = b(S)$, by Corollary~\ref{cor:StGL} $\psi$ must be the Steinberg
character $\St_C$ of $C$, and so   
$$\psi = \psi_1 \otimes \psi_2 \otimes \ldots \otimes \psi_r,$$
where $\psi_i$ is the Steinberg character of $\GL_{k_i}(q^{d_i})$, of degree
$q^{d_ik_i(k_i-1)/2}$.

Observe that $s \neq 1$, i.e. $\chi$ is not unipotent. Otherwise
$b(S) = \St(1) = q^{n(n-1)/2}$. However, the character $\rho \in \Irr(S)$ labeled by 
$((u), \St_{C_S(u)})$, where $u \in S$ is an element of order $3$ with centralizer 
$C_S(u) \cong C_3 \times \GL_{n-2}(2)$, has degree
$$\frac{(q^n-1)(q^{n-1}-1)}{3} \cdot q^{(n-2)(n-3)/2} > q^{n(n-1)/2} = b(S)$$
as $n \geq 4$, a contradiction.

Next we show that $r > 1$. Assume the contrary: $C \cong \GL_k(q^d)$ with $kd = n$
and $d > 1$. Then by Lemma \ref{sum}(ii) we have 
$$\chi(1) = q^{dk(k-1)/2} \cdot \frac{(q-1)(q^2-1) \ldots (q^n-1)}
  {(q^d-1)(q^{2d}-1) \ldots (q^{kd}-1)} < 
  q^{dk(k-1)/2} \cdot \frac{q^{n(n+1)/2-1}}{\frac{9}{16}q^{dk(k+1)/2}} = 
  \frac{8}{9}q^{n(n-1)/2}$$
as $q = 2$. Thus $\chi(1) < \St(1)$, a contradiction.

Thus we must have that $r \geq 2$. Observe that 
there is a semisimple element $t \in S$ with centralizer
$$C_S(t) \cong \GL_1(q^{k_1d_1+k_2d_2}) \times \GL_{k_3}(q^{d_3}) \times \ldots 
  \times \GL_{k_r}(q^{d_r}).$$
Choose $\psi \in \Irr(S)$ to be labeled by $((t),\St_{C_S(t)})$. Then
$$\frac{\psi(1)}{\chi(1)} = 
  \frac{\prod^{k_1}_{i=1}(q^{id_1}-1)\cdot \prod^{k_2}_{i=1}(q^{id_2}-1)}
  {q^{d_1k_1(k_1-1)/2} \cdot q^{d_2k_2(k_2-1)/2} \cdot (q^{k_1d_1+k_2d_2}-1)}.$$  
By Lemma \ref{sum}(ii), 
$1> \prod^{k_j}_{i=1}(q^{id_j}-1)/q^{d_jk_j(k_j+1)/2} > 9/32$ for $j = 1,2$
(in fact we can replace $9/32$ by $9/16$ if $d_j > 1$). Since $(d_1,d_2) \neq (1,1)$,
it follows that $1 > \psi(1)/\chi(1) > 81/512$. 

\medskip
2) Next we consider the case $S = \Sp_{2n}(2)$ with $n \geq 3$.
Then the dual group is again isomorphic to $S$, and we can view 
$S = \cG^F$ for a Frobenius map $F$ of $\cG = \Sp_{2n}(\bar{\F}_2)$. Since 
$Z(\cG) = 1$, by Corollary 14.47 and Proposition 14.42 of \cite{DM}, $S$ has a 
unique Gelfand-Graev character $\Gamma$, which is the sum of $2^n$
\emph{regular} irreducible characters $\chi_{(s)}$. Each such $\chi_{(s)}$
has Lusztig label $((s),\St_{C_S(s)})$, where $(s)$ is any semisimple class
in $S$, see e.g. \cite{H1}.

Note that $\Gamma(1) = |S|_{2'} = \prod^{n}_{i=1}(2^{2i}-1) > 2^{n(n+1)} \cdot (9/16)$, 
with the latter inequality following from Lemma \ref{sum}(ii). Hence by the 
Cauchy-Schwarz inequality we have
$$\sum_{(s)}\chi_{(s)}(1)^2 \geq \frac{(\sum_{(s)}\chi_{s}(1))^2}{2^n}  
  = \frac{(|S|_{2'})^{2}}{2^n} > \frac{9}{16} \cdot |S| 
  \geq \frac{9}{16} \cdot b(S)^2.$$
In particular, $\varep(S) > 9/16$ if $b(S)$ is not achieved by any regular 
character $\chi_{(s)}$. So we will assume that $b(S)$ is achieved by a 
regular character $\chi = \chi_{(s)}$.

According to \cite[Lemma 3.6]{TZ1}, $C:= C_S(s) = D_1 \times \ldots \times D_r$ 
is a direct product of groups of the
form $\GL^{\e}_{k}(q^d)$ (where $\e = +1$ for $\GL$ and $\e = -1$ for $\GU$) or
$\Sp_{2m}(q)$. Note that, since $q=2$, $C$ contains at most one factor of the latter 
form, and no factor of the former form with $(d,\e) = (1,1)$.  
First suppose that all of the factors $D_i$ are of the second form.
It follows that $s = 1$, $\chi_{(1)} = \St$. 
Choosing $\psi \in \Irr(S)$ to be labeled by 
$((u), \St_{C_S(u)})$, where $u \in S$ is an element of order $3$ with centralizer 
$C_S(u) \cong C_3 \times \Sp_{2n-2}(2)$, we see that
$$b(S)= \chi(1) = 2^{n^2} > \psi(1) = \frac{(2^{2n}-1)}{3} \cdot 2^{(n-1)^2} > b(S)/2$$
as $n \geq 2$.

Next we consider the case where exactly one of the factors $D_i$ is of the form
$\GL^{\e}_k(q^d)$. Then, since $q=2$ we must actually have $r \leq 2$,  
$C = \GL^{\e}_{k}(q^d) \times \Sp_{2m}(q)$ with $m := n-kd$, and $(d,\e) \neq (1,1)$.
Hence by Lemma \ref{sum}(ii)  
$$\chi(1) = q^{dk(k-1)/2+m^2} \cdot \frac{\prod^{n}_{j=m+1}(q^{2j}-1)}
  {\prod^{k}_{j=1}(q^{jd}-\e^j)} < 
  q^{dk(k-1)/2+m^2} \cdot \frac{q^{n(n+1)-m(m+1)}}{\frac{9}{16}q^{dk(k+1)/2}} 
  < \frac{16}{9}q^{n^2}.$$
It is easy to check that $\chi(1) \neq q^{n^2}$. Choosing $\psi = \St$, we 
then have $1 > \psi(1)/b(S) > 9/16$. 

Lastly, we consider the case where at least two of the factors $D_i$ are of form
$\GL^{\e}_k(q^d)$:
$$C =  \GL^{\e_1}_{k_1}(q^{d_1}) \times \GL^{\e_2}_{k_2}(q^{d_2}) \times \ldots \times 
       \GL^{\e_{r-1}}_{k_{r-1}}(q^{d_{r-1}}) \times \Sp_{2m}(q),$$
where $r-1 \geq 2$ and $m$ can be zero. We will assume that 
$k_1d_1 \geq k_2d_2 \geq \ldots \geq k_{r-1}d_{r-1} \geq 1$.  
Observe that there is a semisimple element $t \in S$ with centralizer
$$C_S(t) \cong \GU_1(q^{k_1d_1+k_2d_2}) \times \GL^{\e_3}_{k_3}(q^{d_3}) \times \ldots \times
   \GL^{\e_{r-1}}_{k_{r-1}}(q^{d_{r-1}}) \times \Sp_{2m}(q).$$
Choose $\psi \in \Irr(S)$ to be labeled by $((t),\St_{C_S(t)})$. Then
$$\frac{\psi(1)}{\chi(1)} = 
  \frac{\prod^{k_1}_{i=1}(q^{id_1}-\e_1^i)\cdot \prod^{k_2}_{i=1}(q^{id_2}-\e_2^i)}
  {q^{d_1k_1(k_1-1)/2} \cdot q^{d_2k_2(k_2-1)/2} \cdot (q^{k_1d_1+k_2d_2}+1)}.$$  
By Lemma \ref{sum}(ii), 
$\prod^{k_j}_{i=1}(q^{id_j}-\e_j^i)/q^{d_jk_j(k_j+1)/2} > 9/16$ for $j = 1,2$
since $(d_j,\e_j) \neq (1,1)$. Furthermore, since 
$k_1d_1+k_2d_2 \geq 2$ we have $q^{k_1d_1+k_2d_2}+1 \leq (5/4)q^{k_1d_1+k_2d_2}$.
Thus $\psi(1)/\chi(1) > 81/320$, and so we are done if $\psi(1) \neq \chi(1)$.
Suppose that $\psi(1) = \chi(1)$. Then $k_1=k_2=1$, $(d_1,\e_1)=(2,1)$, and
$(d_2,\e_2) = (1,-1)$. In this case we can replace $t$ by a semisimple element $t'$
with  
$$C_S(t') \cong \GL_1(q^{k_1d_1+k_2d_2}) \times \GL^{\e_3}_{k_3}(q^{d_3}) \times \ldots \times
   \GL^{\e_{r-1}}_{k_{r-1}}(q^{d_{r-1}}) \times \Sp_{2m}(q)$$
and repeat the above argument.

\medskip
3) Finally, let us consider the case $S = \Omega^{\e}_{2n}(q)$ with $n \geq 4$. Again,
the dual group of $S$ can be identified with $S$, and we can view 
$S = \cG^F$ for a Frobenius map $F$ of $\cG = \SO_{2n}(\bar{\F}_2)$. Since 
$\cG$ is simply connected (as $q=2$), the centralizers of any 
semisimple element in $\cG$ are connected. Arguing as in the $\Sp$-case,
we may assume that $b(S)$ is attained at a regular character $\chi = \chi_{(s)}$.
One can show (see also \cite[Lemma 3.7]{TZ1}) that
$$C := C_S(s) = K_1 \times H_3 \times \ldots \times H_r$$
where each $H_i$ with $3 \leq i \leq r$ is of the form $\GL^{\beta}_k(q^d)$
with $\beta = \pm 1$.
Furthermore, $K_1$ has a normal subgroup $H_1 \cong \Omega^{\pm}_{2m}(q)$ (where
$m$ can be zero) such that  
$K_1/H_1$ is either trivial, or isomorphic to $\GU_2(2)$. In the latter case, 
the Steinberg character of the finite connected reductive group $K_1$ has degree
equal to $|K_1|_2 = 2^{m(m-1)+1}$. Thus in either case we may replace $C_S(s)$ by
$$H_1 \times H_2 \times H_3 \times \ldots \times H_r$$
(where each $H_i$ is of the form $\GL^{\beta}_k(q^d)$ with $\beta = \pm 1$ or 
$\Omega^{\pm}_{2m}(q)$, and the latter form can occur for at most one factor $H_i$),
and identify the Steinberg character $\St_C$ with 
$\St_{H_1} \otimes \ldots \otimes \St_{H_r}$.

First we consider the case $s = 1$, i.e. $C = H_1 = S$, and $\chi(1) = \St$. 
Choosing $\psi \in \Irr(S)$ to be labeled by 
$((u), \St_{C_S(u)})$, where $u \in S$ is an element of order $3$ with centralizer 
$C_S(u) \cong C_3 \times \Omega^{-\e}_{2n-2}(2)$, we see that
$$b(S)= \chi(1) = 2^{n(n-1)} > \psi(1) = \frac{(2^{n}-\e)(2^{n-1}-\e)}{3} \cdot 
  2^{(n-1)(n-2)} > b(S)/4$$
as $n \geq 4$.

Next we consider the case where exactly one of the factors $H_i$ is of the form
$\GL^{\al}_k(q^d)$. Then, since $q=2$ we must actually have $r \leq 2$,  
$C = \GL^{\al}_{k}(q^d) \times \Omega^{\beta}_{2m}(q)$ with $m := n-kd$, and 
$(d,\al) \neq (1,1)$.
Hence by Lemma \ref{sum}(ii)  
\begin{align*}
\chi(1) & = 
  q^{dk(k-1)/2+m(m-1)} \cdot \dfrac{(q^m+\beta) \cdot \prod^{n}_{j=m+1}(q^{2j}-1)}
  {(q^n+\e) \cdot \prod^{k}_{j=1}(q^{jd}-\al^j)}\\
  & < q^{dk(k-1)/2+m(m-1)} \cdot \dfrac{3}{2} \cdot \dfrac{16}{15} \cdot 
    \dfrac{q^{m+ n(n+1)-m(m+1)}}{\frac{9}{16}q^{n+ dk(k+1)/2}} 
  < \frac{128}{45}q^{n(n-1)}.
\end{align*}
It is easy to check that $\chi(1) \neq q^{n(n-1)}$. Choosing $\psi = \St$ we 
then have $1 > \psi(1)/b(S) > 45/128$. 

Lastly, we consider the case where at least two of the factors $H_i$ are of the
form $\GL^{\pm}_k(q^d)$:
$$C =  \GL^{\e_1}_{k_1}(q^{d_1}) \times \GL^{\e_2}_{k_2}(q^{d_2}) \times 
       \ldots \times \GL^{\e_{r-1}}_{k_{r-1}}(q^{d_{r-1}}) \times \Omega^{\beta}_{2m}(q),$$
where $r-1 \geq 2$ and $m$ can be zero. We will assume that 
$k_1d_1 \geq k_2d_2 \geq \ldots \geq k_{r-1}d_{r-1} \geq 1$.  
Observe that there is a semisimple element $t \in S$ with centralizer
$$C_S(t) \cong \GL^{\al}_1(q^{k_1d_1+k_2d_2}) \times \GL^{\e_3}_{k_3}(q^{d_3}) \times \ldots 
  \times \GL^{\e_{r-1}}_{k_{r-1}}(q^{d_{r-1}}) \times \Omega^{\beta}_{2m}(q)$$
for some $\al = \pm 1$. Choose $\psi \in \Irr(S)$ to be labeled by 
$((t),\St_{C_S(t)})$. Then
$$\frac{\psi(1)}{\chi(1)} = 
  \frac{\prod^{k_1}_{i=1}(q^{id_1}-\e_1^i)\cdot \prod^{k_2}_{i=1}(q^{id_2}-\e_2^i)}
  {q^{d_1k_1(k_1-1)/2} \cdot q^{d_2k_2(k_2-1)/2} \cdot (q^{k_1d_1+k_2d_2}-\al)}.$$  
By Lemma \ref{sum}(ii), 
$\prod^{k_j}_{i=1}(q^{id_j}-\e_j^i)/q^{d_jk_j(k_j+1)/2} > 9/16$ for $j = 1,2$
since $(d_j,\e_j) \neq (1,1)$. Furthermore, since 
$k_1d_1+k_2d_2 \geq 2$ we have $q^{k_1d_1+k_2d_2}-\al \leq (5/4)q^{k_1d_1+k_2d_2}$.
Thus $\psi(1)/\chi(1) > 81/320$, and so we are done if $\psi(1) \neq \chi(1)$.
Suppose that $\psi(1) = \chi(1)$. Then $k_1=k_2=1$, which forces $\al = \e_1\e_2$,
$(d_1,\e_1)=(2,1)$, and $(d_2,\e_2) = (1,-1)$. In this last case we must have
that $r=3$, 
$$C = \GL_1(4) \times \GU_1(2) \times \Omega^{-\e}_{2n-6}(2),$$
and 
$$\chi(1) = \frac{1}{9} \cdot 
  2^{(n-3)(n-4)}(2^n-\e)(2^{n-3}-\e)(2^{2n-2}-1)(2^{2n-4}-1).$$
In particular, $1 \neq \chi(1)/\St(1) < 4/3$, and so we are done.       
\end{proof}

\section{The Largest Degrees of Simple Groups of Lie Type}
Let $L$ be a finite Lie-type group of simply connected type over $\F_q$.
When $q$ is large enough in comparison to the rank of $L$, Theorem~\ref{bound1}
gives us the precise value of $b(L)$. However, we do not have a formula for $b(L)$ for small values of $q$. In the extreme case $L = \SL_n(2)$, there does not even
seem to exist  a decent upper bound on 
$b(L)$ in the literature, aside from the trivial bound $b(L) < |L|^{1/2}$. 
On the other hand, as a polynomial of $q$, the degree of the Steinberg
character $\St$ is the same as that of the bound in Theorem~\ref{bound1}. So it is an
interesting question to study the asymptotic of the quantity 
$c(L) := b(L)/\St(1)$.
In this section we will prove upper and lower bounds for $c(L)$ for 
finite classical groups.

\begin{thm}   \label{bound4}
 Let $G$ be any of the following Lie-type groups of type $A$: 
 $\GL_n(q)$, $\PGL_n(q)$, $\SL_n(q)$, or $\PSL_n(q)$. 
 Then the following inequalities hold:
 $$\max\left\{1, \frac{1}{4}
           \left(\log_{q}((n-1)(1-\frac{1}{q})+q^2)\right)^{3/4}\right\} 
  \leq \frac{b(G)}{q^{n(n-1)/2}} < 13(\log_q(n(q-1)+q))^{2.54}.$$
 In particular,
 $$\frac{1}{4}\left(\log_{q}\frac{n+7}{2}\right)^{3/4} < \frac{b(G)}{q^{n(n-1)/2}}
   < 13(1+\log_q(n+1))^{2.54}.$$
\end{thm}

\begin{proof}  
1) Since the Steinberg character of $\GL_n(q)$ is trivial at $Z(\GL_n(q))$
and stays irreducible as a character of $\PSL_n(q)$, the inequality 
$b(G) \geq q^{n(n-1)/2}$ is obvious. Next we prove the upper bound 
$$c(G) := \frac{b(G)}{q^{n(n-1)/2}} < 13(\log_q(n(q-1)+q))^{2.54}$$
for $G = \GL_n(q)$, which then also implies the same bound for all other
groups of type $A$. It is not hard to see that the arguments in p. 1) of the
proof of Theorem~\ref{clas2} also carry over to
the case of $G = \GL_n(q)$. It follows that $c(G)$ is just 
the maximum of
$$P := \frac{\prod^{n}_{i=1}(1-q^{-i})}
  {\prod^m_{j=1}\prod^{k_j}_{i=1}(1-q^{-id_j})},$$
where the maximum is taken over all possible $m,k_j,d_j \geq 1$ with 
$k_1d_1 \geq \ldots \geq k_md_m$, 
$\sum^{m}_{j=1}k_jd_j = n$, and for each $d = 1,2, \ldots$, 
the number $a_d$ of the values of $j$ such that $d_j$ equals to $d$ does 
not exceed the number $\fp_d$ of monic irreducible polynomials $f(t)$ of 
degree $d$ over $\F_q$; in particular, $a_d < q^d/d$.

By Lemma \ref{sum}(ii), the numerator of $P$ is bounded between 
$9/32$ and $1$ for all $q \geq 2$. It remains to bound (the natural logarithm 
$L$ of) the denominator of $P$.
By Lemma \ref{sum}(i), $\prod^{\infty}_{i=1}(1-q^{-id_j}) > \exp(-\al q^{-d_j})$ with
$\al = 2\ln(32/9)$. Hence,
$$-L/\al < \sum_{j=1}^{m} q^{-d_j} = \sum^{n}_{d=1}a_d q^{-d}.$$
The constraints imply $\sum^n_{d=1}  d a_d \leq n$ and $a_d < q^d/d$.
Replacing the $a_d$ with real numbers $x_d$, we want to
maximize $\sum^{\infty}_{d=1}x_dq^{-d}$  subject to the constraints
$$\sum^{\infty}_{d=1}d x_d \leq n, ~~0 \leq x_d \leq q^d/d.$$
Since the function $q^{-t}/t$ is decreasing on $(0,\infty)$, 
we see that there exists some $d_0$ (depending on $n$) such that the sum is 
optimized when $x_i = q^i/i$ for all $i<d_0$ and $x_i = 0$ for all $i>d_0$.  
Thus $d_0$ is the largest integer such that
$\sum^{d_0-1}_{d=1}(q^d/d)d = \sum^{d_0-1}_{d=1}q^d = (q^{d_0}-q)/(q-1)$ does not 
exceed $n$, whence
$$d_0 \leq \log_q(n(q-1)+q)  < 1+ \log_q(n +1).$$    
On the other hand,
$$\sum^{d_0}_{d=1}x_d q^{-d} \leq \sum_{d=1}^{d_0} \frac{1}{d} < 1 + \ln(d_0).$$
Thus $L > -\al(1+\ln(d_0))$ and so
$$P < e^{-L} < e^{\al(1+\ln(d_0))} = e^{\al}d_{0}^{\al} < 13(\log_q(n(q-1)+q))^{2.54}$$
by the choice $\al = 2\ln(32/9)$.

\medskip
2) Now we prove the lower bound 
$$c(S) := \frac{b(S)}{q^{n(n-1)/2}} > \frac{1}{4}
           \left(\log_{q}((n-1)(1-\frac{1}{q})+q^2)\right)^{3/4}$$
for $S = \PSL_n(q)$, which then also implies the same bound for all other 
groups of type $A$. As above, let $\fp_d$ be the number of monic 
irreducible polynomials $f(t)$ over $\F_q$. Arguing as in p. 1) of the proof 
of Theorem \ref{clas}, we see that the total number of elements
of $\F_{q^d}$ which do not belong to any proper subfield of $\F_{q^d}$ is 
at least $3q^d/4$ when $d \geq 3$ and at most $q^d-1$. It follows that
for $d \geq 3$ we have 
\begin{equation}\label{poly}
  \frac{3q^d}{4d} \leq \fp_d < \frac{q^d}{d}.
\end{equation}

Since $b(S) \geq \St(1)$, the claim is obvious if $n \leq q^3$.
Hence we may assume that $n \geq q^3+1 \geq 3\fp_3+3$. Let $d^* \geq 3$ 
be the largest integer such that $m:= \sum^{d^*}_{d=3}d\fp_d \leq n-3$. 
In particular, 
$$\sum^{d^*+1}_{d=3}q^d > \sum^{d^*+1}_{d=3}d \fp_d \geq n-2,$$ 
and so 
\begin{equation}\label{for-d1}
  d^*+1 \geq \log_q((n-1)(1-1/q)+q^2).
\end{equation} 
Observe that $G_1 := \GL_m(q)$ contains a semisimple element $s_1$ with
$$C_{G_1}(s_1) = \GL_{1}(q^3)^{\fp_3} \times \GL_{1}(q^4)^{\fp_4} \times 
  \ldots \times \GL_{1}(q^{d^*})^{\fp_{d^*}}.$$
(Indeed, each of the $\fp_d$ monic irreducible polynomials of degree $d$ over $\F_q$
gives us an embedding $\GL_1(q^d) \hookrightarrow \GL_d(q)$.) 
If $\det(s_1) = 1$, then choose $s := \diag(I_{n-m},s_1)$ so that 
$$C_G(s) = \GL_{n-m}(q) \times C_{G_1}(s_1).$$
Otherwise we choose $s := \diag(I_{n-m-1},\det(s_1)^{-1},s_1)$ so that 
$$C_G(s) = \GL_{n-m-1}(q) \times \GL_1(q) \times C_{G_1}(s_1).$$
In either case, $\det(s) = 1$ and so $s \in [G,G]$ for $G = \GL_n(q)$.
Now consider the 
(regular) irreducible character $\chi$ labeled by $((s),\St_{C_G(s)})$. 
The inclusion $s \in [G,G]$ implies that $\chi$ is trivial at $Z(G)$.
Also, our choice of $s$ ensures that $s$ has at most two eigenvalues in
$\F_q^{\times}$: the eigenvalue $1$ with multiplicity $\geq n-m-1 \geq 2$, and
at most one more eigenvalue with multiplicity $1$. Hence, for any 
$1 \neq t \in \F_q^{\times}$, $s$ and $st$ are not conjugate in $G$. To
each such $t$ one can associate a linear character $\hat{t} \in \Irr(G)$ 
in such a way that the multiplication by $\hat{t}$ yields a bijection
between the Lusztig series ${\mathcal E}(G,(s))$ and 
${\mathcal E}(G,(st))$ labeled by the 
conjugacy classes of $s$ and $st$, cf. \cite[Proposition 13.30]{DM}. Since 
the distinct Lusztig series are disjoint, we conclude that the number of 
linear characters $\hat{t} \in \Irr(G)$ such that $\chi\hat{t} = \chi$ is
exactly one. It then follows by \cite[Lemma 3.2(i)]{KT} that $\chi$ is 
irreducible over $\SL_n(q)$.
Thus we can view $\chi$ as an irreducible character of $S = \PSL_n(q)$.   

Next, in the case $\det(s_1) = 1$ we have  
$$\frac{\chi(1)}{q^{n(n-1)/2}} = 
  \frac{\prod^{n}_{i=n-m+1}(1-q^{-i})}
  {\prod^{d^*}_{j=3}(1-q^{-j})^{\fp_j}},$$ 
whereas in the case $\det(s_1) \neq 1$ we have that  
$$\frac{\chi(1)}{q^{n(n-1)/2}} =  \frac{\prod^{n}_{i=n-m}(1-q^{-i})}
  {(1-q^{-1})\prod^{d^*}_{j=3}(1-q^{-j})^{\fp_j}}.$$
Since $n-m \geq 3$, in either case we have 
$$\frac{\chi(1)}{q^{n(n-1)/2}} >
  \frac{\prod^{\infty}_{i=4}(1-q^{-i})}
  {\prod^{d^*}_{j=3}(1-q^{-j})^{\fp_j}}.$$  
By Lemma \ref{sum}(ii), the numerator is at least 
$(9/16) \cdot (4/3) \cdot (8/7) = 6/7$. To estimate the 
denominator, observe that $1/(1-x) > e^{x}$ for $0 < x < 1$. 
Applying (\ref{poly}) we now see that
$$\begin{aligned}
   \ln\left(\frac{1}{\prod^{d^*}_{j=3}(1-q^{-j})^{\fp_j}}\right) &> 
   \sum^{d^*}_{j=3}q^{-j}\fp_j \geq  \sum^{d^*}_{j=3}\frac{3}{4j}\\
  &\geq \frac{3}{4}(\ln(d^*+1) - 1 - \frac{1}{2})
   = \frac{3\ln(d^*+1)}{4} - \frac{9}{8}.
\end{aligned}$$ 
Together with (\ref{for-d1}) this implies that 
$$\frac{\chi(1)}{q^{n(n-1)/2}} > 
  \frac{6}{7e^{9/8}}\left(\log_q((n-1)(1-\frac{1}{q})+q^2)\right)^{3/4}\!\!
  > \frac{1}{4}\left(\log_{q}((n-1)(1-\frac{1}{q})+q^2)\right)^{3/4}\!\!.$$
\end{proof}

Next we handle the other classical groups.
Abusing the notation, by \emph{a group of type $C_n$ over $\F_q$} we mean
any of the following groups: $\Sp_{2n}(q)$ (of simply connected type),
$\PCSp_{2n}(q)$ (of adjoint type), or $\PSp_{2n}(q)$ (the simple group, except
for a few ``small'' exceptions). Similarly, by  
\emph{a group of type $B_n$ over $\F_q$} we mean any of the following group: 
$\Spin_{2n+1}(q)$ (of simply connected type), $\SO_{2n+1}(q)$ (of adjoint
type), or $\Omega_{2n+1}(q)$ (the simple group, except
for a few ``small'' exceptions). By \emph{a group of type $D_n$ or 
$^2 D_n$ over $\F_q$} we mean any of the following group: 
$\Spin^{\e}_{2n}(q)$ (of simply connected type),
$\PCO^{\e}_{2n}(q)^{\circ}$ (of adjoint type), 
$\POm^{\e}_{2n}(q)$ (the simple group, except
for a few ``small'' exceptions), $\SO^{\e}_{2n}(q)$, as well as
the half-spin group $\HS_{2n}(q)$. We refer the reader
to \cite{C} for the definition of these finite groups of Lie type.

\begin{thm}   \label{bound5}
 Let $G$ be a group of type $B_n$, $C_n$, $D_n$, or $^2D_n$ over $\F_q$.
If $q$ is odd, then the following inequalities hold:
$$\max\left\{1, \frac{1}{5}
           \left(\log_{q}\frac{4n+25}{3}\right)^{3/8}\right\} 
  \leq \frac{b(G)}{\St(1)} < 38(1+\log_q(2n+1))^{1.27}.$$
If $q$ is even, then the following inequalities hold:
$$\max\left\{1, \frac{1}{5}
           \left(\log_{q}(n+17)\right)^{3/8}\right\} 
  \leq \frac{b(G)}{\St(1)} < 8(1+\log_q(2n+1))^{1.27}.$$
\end{thm} 

\begin{proof}
In all cases, the bound $b(G)/\St(1) \geq 1$ is obvious since $\St$ 
is irreducible over any of the listed possibilities for $G$. 

\medskip
1) We begin by proving the upper bound in 
the cases where $G$ is of type $B_n$, respectively of type
$D_n$ or $^2D_n$, and $q$ is odd and $n \geq 3$. Let $V = \F_q^{2n+1}$, respectively 
$V = \F_q^{2n}$, be endowed with a non-degenerate quadratic form. Then it is 
convenient to work with the \emph{special Clifford group} 
$G := \Gamma^+(V)$ associated to the quadratic space $V$, see for instance 
\cite{TZ2}. In particular, $G$ maps onto $\SO(V)$ with kernel $C_{q-1}$, and 
contains $\Spin(V)$ as a normal subgroup of index $q-1$. Furthermore,
the dual group $G^*$ can be identified with the conformal symplectic group 
$\CSp_{2n}(q)$ in the $B$-case, and with the group
$\CO(V)^{\circ}$ in the $D$-case, cf. \cite[\S3]{FS}. Observe that  
the adjoint group $\PCO(V)^{\circ}$ contains $\PSO(V)$ 
as a normal subgroup of index $2$. Similarly,
the half-spin group contains a quotient of $\Spin(V)$ (by a central subgroup
of order $2$) as a normal subgroup of index $2$.
Hence, it suffices to prove the indicated upper bound 
(with constant $19$) for this particular $G$. 
Similarly, it will suffice to prove the 
indicated lower bound (with constant $1/5$) for the simple group
$S = \POm(V)$. 

\smallskip
Let $s \in G^*$ be any semisimple element. Consider for instance the $B$-case
and let $\tau(s) \in \F_q^{\times}$ 
denote the factor by which the conformal transformation
$s \in \CSp_{2n}(q)$ changes the corresponding symplectic form. Also set
$H := \Sp_{2n}(q)$ and denote by $\F_q^{\times 2}$ the set of squares
in $\F_q^{\times}$. Then by \cite[Lemma 2.4]{N} we have that 
$$C := C_{G^*}(s) = C_{H}(s) \cdot C_{q-1},$$ with
$$C_{H}(s) = \prod_{i}\GL^{\e_i}_{k_i}(q^{d_i}) \times 
  \left\{ \begin{array}{lll} \Sp_l(q^2), & 
            \tau(s) \notin \F_q^{\times 2}, & (B1)\\ 
    \Sp_{2k}(q) \times \Sp_{2l-2k}(q), & 
            \tau(s) \in \F_q^{\times 2}, & (B2), \end{array}\right.$$
where $\sum_ik_id_i = n-l$, $\e_i = \pm 1$, and $0 \leq k \leq l \leq n$
(and we use $(B1)$ and $(B2)$ to label the two subcases which can arise). 
In the $D$-case, let $\tau(s) \in \F_q^{\times}$ 
denote the factor by which the conformal transformation
$s \in \CO(V)^{\circ}$ changes the corresponding quadratic form; also set
$H := \SO(V)$. Then by \cite[Lemma 2.5]{N} we have that 
$$C := C_{G^*}(s) = C_{H}(s) \cdot C_{q-1},$$ with
$$C_{\GO(V)}(s)\cong \prod_{i}\GL^{\e_i}_{k_i}(q^{d_i}) \times 
  \left\{ \begin{array}{lll} \GO^{\pm}_l(q^2), &  
            \tau(s)\notin \F_q^{\times 2}, & (D1)\\ 
    \GO^{\pm}_{2k}(q) \times \GO^{\pm}_{2l-2k}(q), & 
            \tau(s) \in \F_q^{\times 2}, & (D2),
\end{array}\right.$$
where $\sum_ik_id_i = n-l$, $\e_i = \pm 1$, and $0 \leq k \leq l \leq n$
(and we use $(D1)$ and $(D2)$ to label the two subcases which can arise).

\smallskip
On the set of monic irreducible polynomials of degree $d$ in $\F_q[t]$
(regardless of whether $q$ is odd or not), one
can define the involutive map $f \mapsto \check{f}$ such that $x^df(1/x)$
is a scalar multiple of $f(x)$. One can
show that such an $f$ can satisfy the equality $f = \check{f}$ only when 
$2|d$ and $\al^{q^{d/2}+1} = 1$ for every root $\al$ of $f$. Hence, if 
$\fps_d$ denotes the number of monic irreducible polynomials of degree 
$d$ over $\F_q$ with $f \neq \check{f}$, then
\begin{equation}\label{poly2}
  \fps_d < \frac{q^d}{d}, \mbox{ and }\fps_d \geq \frac{3q^d}{4d} 
  \mbox{ if }d \geq 3 \mbox{ and }q \geq 3, \mbox { or if } d \geq 5
  \mbox{ and }q = 2.
\end{equation}
The former inequality is obvious. The latter inequality follows from 
(\ref{poly2}) when $d$ is odd (as $\fps_d = \fp_d$ in this case), 
and by direct check when $d = 4$ or $(q,d) = (2,6)$. Assume 
$d = 2e \geq 6$ and $(q,d) \neq (2,6)$. Then the number of elements of
$\F_{q^d}$ which belong to a proper subfield of $\F_{q^d}$ or to the subgroup
$C_{q^{d/2}+1}$ of $\F_{q^{d}}^{\times}$ is at most
$$q^e + \sum^e_{i=1}q^i < q^e(q+1) < q^d/4,$$
whence $\fps_d > 3q^d/4$ as stated. 

In either case, decompose the characteristic polynomial of the transformation
$s$ into a product of powers of distinct monic irreducible polynomials over
$\F_q$. Then the factor $GL^{\e_i}_{k_i}(q^{d_i})$ in $C$ with $\e_i = 1$, 
respectively with $\e_i = -1$, corresponds to a factor $(f_i\check{f}_i)^{k_i}$
in this decomposition with $\deg(f_i) = d_i$ and $f_i \neq \check{f}_i$, 
respectively to a factor $f_i^{k_i}$ in this decomposition with 
$\deg(f_i) = 2d_i$ and $f_i = \check{f}_i$. In particular, if $a_d$ denotes the
number of factors $GL_{k_i}(q^d_i)$ in $C$ with $d_i = d$, then 
\begin{equation}\label{for-ad}
   a_d \leq \frac{\fps_d}{2} < \frac{q^d}{2d}, ~~~\sum_dda_d \leq n.
\end{equation}      

\smallskip
Certainly, $\St_C(1) = |C|_{p}$ for the prime $p$ dividing $q$. But, since 
the centralizer of $s$ in the corresponding algebraic group is (most of the time)
disconnected, we cannot apply Theorem \ref{thm:steinberg} directly to say that  
$b(G)$ is attained by the regular character $\chi=\chi_{(s)}$ labeled by 
$((s),\St_C)$. Nevertheless, we claim that the degree of any unipotent character
$\psi$ of $C$ is at most $2\kappa\cdot\St_C(1)$ and so $b(G) \leq 2\kappa\chi(1)$,
where $\kappa = 2$ in the $(D2)$-case and $\kappa = 1$ otherwise. Indeed,
by definition the unipotent character $\psi$ of the (usually disconnected) group
$C$ is an irreducible constituent of $\Ind^{C}_{D}(\varphi)$ for some 
unipotent character $\varphi$ of $D := Z(G^*)C_H(s)$. In turn, $\varphi$ 
restricts irreducibly to a unipotent character of the (usually disconnected)
group $C_H(s)$. It is easy to see that $C_H(s)$ contains a normal subgroup $D_1$ 
of index $\kappa$, which is a finite connected group, in fact a 
direct product of subgroups of form $\GL^{\e_i}_{k_i}(q^{d_i})$, 
$\Sp_l(q^2)$, $\Sp_l(q)$, $\SO^{\pm}_l(q^2)$, or $\SO^{\pm}_l(q)$. 
Again by definition
$\varphi|_{C_H(s)}$ is an irreducible constituent of $\Ind^{C_H(s)}_{D_1}(\varphi_1)$ 
for some unipotent character $\varphi_1$ of $D_1$. Now we can apply
Theorem \ref{thm:steinberg} to $D_1$ to see that 
$\varphi_1(1) \leq \St_{D_1}(1) = |D_1|_p = |C|_p$. Since $|C/D| = 2$ and 
$|C_H(s)/D_1| = \kappa$, we conclude that $\psi(1) \leq 2\kappa|C|_p$, as stated.

Observe that $(G^*:C)_{p'} = (H:D_1)_{p'}/\kappa$. We have therefore shown
that 
$$b(G) = \chi(1) \leq 2(H:D_1)_{p'}\cdot |D_1|_p.$$
By Lemma \ref{sum}(i), (iv),
$\prod^{\infty}_{i=1}(1-q^{-2i}) > 71/81$ since $q \geq 3$, and  
$\prod^{k_j}_{i=1}(1-(\e_jq^{-d_j})^{i}) > 1$ if $\e_j = -1$.
Furthermore, $q^l \pm 1 \geq (2/3)q^l$ and $q^n \pm 1 \leq (28/27)q^n$ since 
$n,q \geq 3$. Using these estimates, we see that
\begin{equation}\label{for-cg1}
  c(G) \leq \frac{A}{\prod_{j~:~\e_j = 1}\prod^{k_j}_{i=1}(1-q^{-id_j})}.
\end{equation}  
Here, $A = 2 \cdot (28/27) \cdot (81/71) \cdot (3/2)^2 = 378/71$ in the 
$(D2)$-case. Similarly, $A = 2$ in the $(B1)$-case, $A = 162/71$ in the
$(B2)$-case, and $A = 28/9$ in the $(D1)$-case. By Lemma~\ref{sum}(i),
$c(G) \leq A\cdot \exp(\al\sum_{d}a_dq^{-d})$ with $\al = 2\ln(32/9)$,
and $a_d$ is subject to the constraints (\ref{for-ad}). Now we can argue as
in p. 1) of the proof of Theorem \ref{bound4} to bound $\sum_da_dq^{-d}$ from
above. In particular, we get $\sum_da_dq^{-d} \leq (1+\ln(d_0))/2$, where 
$d_0$ is the largest integer such that $\sum^{d_0-1}_{d=1}d(q^d/2d) \leq n$,
i.e.
$$d_0 \leq \log_q(2n(q-1)+q) < 1 + \log_q(2n+1).$$
Putting everything together, we obtain
\begin{equation}\label{for-cg2}
  c(G) \leq Ae^{\al/2}d_0^{\al/2} < Ae^{1.27}(1+\log_q(2n+1))^{1.27}
\end{equation}
and so we are done, as $Ae^{1.27} < 19$.

\medskip
2) Next we briefly discuss how one can prove the upper bound in the remaining
cases. 

\smallskip
2a) Consider the case $G$ is of type $C_n$ over $\F_q$ with $q$ odd. As above, it 
suffices to prove the upper bound with the constant $19$ for $G = \Sp_{2n}(q)$.
In this case, $G^* = \SO_{2n+1}(q)$, and if $s \in G^*$ is a semisimple element,
then the structure of $C_{\GO_{2n+1}(q)}(s)$ is as described in the $(D2)$-case.
Arguing as above, we arrive at (\ref{for-cg1}) and (\ref{for-cg2}) with 
$Ae^{1.27} = 2 \cdot (81/71) \cdot (3/2)^2 \cdot e^{1.27} < 18.3$.

\smallskip
2b) Next suppose that $G$ is of type $C_n$ over $\F_q$ with $q$ even. 
In this case, $G^* \cong \Sp_{2n}(q)$, and if $s \in G^*$ is a semisimple element,
then the structure of $C_{G^*}(s)$ is as described in the $(B2)$-case with
$k=0$. Arguing as above, we arrive at (\ref{for-cg1}) and (\ref{for-cg2}) with 
$Ae^{1.27} = e^{1.27} < 3.6$.

\smallskip
2c) Finally, let $G = \Omega^{\e}_{2n}(q)$ with $q$ even and $n \geq 4$; 
in particular, $G^* \cong G$. If $s \in G^*$ is a semisimple element,
then the structure of $C_{\GO^{\e}_{2n}(q)}(s)$ is as described in the $(D2)$-case with
$k=0$. Arguing as above and using the estimates $q^l \pm 1 \geq q^l/2$ and 
$q^n \pm 1 \leq (17/16)q^n$, we arrive at (\ref{for-cg1}) and (\ref{for-cg2}) with 
$Ae^{1.27} = 2 \cdot (17/16) \cdot e^{1.27} < 7.6$ for $q \geq 4$. As in p. 3) of
the proof of Theorem \ref{clas2}, in the case $q=2$ we need some extra care
if $C_G(s)$ contains a factor 
$K_1 := ((C_3 \times C_3) \times \Omega^{\pm}_{2r}(2)) \cdot 2$, 
where $C_3 \times C_3$ is the (unique) subgroup of index $2$ in $\GU_2(2)$. 
We claim that we still have the bound $\theta(1) \leq |K_1|_2$ for any unipotent
character $\theta$ of $K_1$. Indeed, $K_1$ is a normal subgroup of index $2$
of $\tilde{K}_1 := \GU_2(2) \times \GO^{\pm}_{2r}(2)$. Now $\theta$ is an 
irreducible constituent of some unipotent character 
$\tilde{\theta} = \lambda \otimes \mu$ of $\tilde{K}_1$, where 
$\lambda \in \Irr(\GU_2(2))$ and $\mu \in \Irr(\GO^{\pm}_{2r}(2))$ are unipotent. 
It follows that the irreducible constituents of $\theta|_{\Omega^{\pm}_{2r}(2)}$ are
unipotent characters of $H_1 := \Omega^{\pm}_{2r}(2)$ and so have degree at most
$\St_{H_1}(1)$ by Theorem \ref{thm:steinberg}.   
But $C_3 \times C_3$ is abelian, so 
$\theta(1) \leq 2 \cdot \St_{H_1}(1) = |K_1|_2$. 
Now we can proceed as in the case $q \geq 4$.

\medskip
3) Now we proceed to establish the logarithmic lower bound for the simple 
groups $S$ of type $D_n$ or $^2D_n$ over $\F_q$ with $q$ odd and $n \geq 4$.
It is convenient to work instead with $G := \SO^{\e}_{2n}(q)$, since
$G^* \cong G$. Since the lower bound is obvious when $n \leq q^3$, we will
assume that $n > q^3 > 3\fps_3+2$. Let $d^* \geq 3$ be the largest integer
such that $m:= \sum^{d^*}_{d=3}d(\fps_d/2) \leq n-2$. In particular, 
$$\sum^{d^*+1}_{d=3}\frac{q^d}{2} > \sum^{d^*+1}_{d=3}d(\fps_d/2) \geq n-1,$$ 
and so 
\begin{equation}\label{for-d2}
  d^*+1 \geq \log_q((2n-1)(1-1/q)+q^2).
\end{equation} 
Observe that $G_1 := \SO^{+}_{2m}(q)$ contains a semisimple element $s_1$ with
$$C_{G_1}(s_1) = \GL_{1}(q^3)^{\fps_3/2} \times \GL_{1}(q^4)^{\fps_4/2} \times 
  \ldots \times \GL_{1}(q^{d^*})^{\fps_{d^*}/2}.$$
(Indeed, each of the $\fps_d$ monic irreducible polynomials $f$ of degree $d$ over 
$\F_q$ with $f \neq \check{f}$ 
gives us an embedding $\GL_1(q^d) \hookrightarrow \SO^{+}_{2d}(q)$.) 
If $s_1 \in \Omega^+_{2m}(q)$, then choose $s := \diag(I_{2n-2m},s_1)$ so that 
$$C_G(s) = \SO^{\e}_{2n-2m}(q) \times C_{G_1}(s_1).$$
Suppose for the moment that $s_1 \notin \Omega^+_{2m}(q)$. Note that there is some 
$\delta \in \F_{q^2} \setminus (C_{q+1} \cup \F_{q})$ such that 
$h \neq \check{h}$ for the minimal (monic) polynomial $h \in \F_q[t]$ of 
$\delta$ and moreover the $\F_q$-norm of $\delta$ is a non-square in $\F_q^{\times}$.
Hence by \cite[Lemma 2.7.2]{KL}, under the embedding 
$\GL_{1}(q^2) \hookrightarrow \GL_2(q) \hookrightarrow \SO^+_4(q)$,
$\delta$ gives rise to an element $s_2 \in \SO^+_4(q) \setminus \Omega^+_4(q)$.   
Now we choose $s := \diag(I_{2n-2m-4},s_2,s_1)$ so that 
$$C_G(s) = \SO^{\e}_{2n-2m-4}(q) \times \GL_1(q^2) \times C_{G_1}(s_1).$$
Our construction ensures that $s \in [G,G] = \Omega^{\e}_{2n}(q)$.

Next we consider the 
(regular) irreducible character $\rho$ labeled by $((s),\St_{C_G(s)})$. 
The inclusion $s \in [G,G]$ implies that $\rho$ is trivial at $Z(G)$. Since 
$S = \POm^{\e}_{2n}(q)$ is a normal subgroup of index $2$ in 
$G/Z(G)$, we see that $S$ has an irreducible character $\chi$ of degree at least
$\rho(1)/2$. Hence in the case $s_1 \in \Omega^+_{2m}(q)$ we have  
$$\frac{\chi(1)}{\St(1)} \geq \frac{1}{2} \cdot 
  \frac{\prod^{n-1}_{i=n-m}(1-q^{-2i}) \cdot (1-\e q^{-n})}
  {\prod^{d^*}_{j=3}(1-q^{-j})^{\fps_j/2} \cdot (1-\e q^{m-n})},$$ 
whereas in the case $s_1 \notin \Omega^{+}_{2m}(q)$ we have that  
$$\frac{\chi(1)}{\St(1)} \geq \frac{1}{2} \cdot 
  \frac{\prod^{n-1}_{i=n-m-2}(1-q^{-2i}) \cdot (1-\e q^{-n})}
  {\prod^{d^*}_{j=3}(1-q^{-j})^{\fps_j/2} \cdot (1-\e q^{m-n+2}) \cdot(1-q^{-2})}$$
(with the convention that $1-\e q^{m-n+2} = 1$ when $m=n-2$). Observe that 
$(1-\e q^{-n})/(1-\e q^{-k}) > q/(q+1) \geq 3/4$ for $0 \leq k \leq n$. Furthermore,
since $n \geq m+2$ we have 
$$\prod^{n-1}_{i=n-m}(1-q^{-2i}) > \frac{\prod^{n-1}_{i=n-m-2}(1-q^{-2i})}{1-q^{-2}} 
  > \prod^{\infty}_{i=2}(1-q^{-2i}) > \frac{71}{72}$$
by Lemma \ref{sum}(i). Thus  
\begin{equation}\label{for-cs1}
  c(S) \geq \frac{\chi(1)}{\St(1)} >
  \frac{B}{\prod^{d^*}_{j=3}(1-q^{-j})^{\fps_j/2}},
\end{equation}
with $B = (71/72) \cdot (3/4) \cdot (1/2) = 71/192$.  
Applying (\ref{poly2}) we now see that
$$\ln\left(\frac{1}{\prod^{d^*}_{j=3}(1-q^{-j})^{\fps_j/2}}\right) > 
  \sum^{d^*}_{j=3}q^{-j}\frac{\fps_j}{2} \geq  \sum^{d^*}_{j=3}\frac{3}{8j} 
  \geq \frac{3\ln(d^*+1)}{8} - \frac{9}{16}.$$ 
Together with (\ref{for-d2}) this implies that 
\begin{equation}\label{for-cs2}
  c(S) > 
  \frac{B}{e^{9/16}}\left(\log_q((2n-1)(1-\frac{1}{q})+q^2)\right)^{3/8}
  > \frac{B}{e^{9/16}}\left(\log_{q}\frac{4n+25}{3}\right)^{3/8},
\end{equation}
and so we are done, since $Be^{-9/16} > 1/5$.
 
\medskip
4) We will now briefly discuss how one can prove the lower bound in the remaining
cases. Notice that the lower bound is obvious when $n \leq q^6$, so 
we will assume $n > q^6$. 

\smallskip
4a) Consider the case $G$ is of type $C_n$ over $\F_q$ with $q$ odd; in particular,
$G^* = \SO_{2n+1}(q)$. Choose $d^* \geq 3$ largest possible such that
$m := \sum^{d^*}_{d=3}d\fps_d/2 \leq n-2$, and so (\ref{for-d2}) holds. Also
choose $s_1 \in G_1 := \SO_{2m+1}(q)$ a semisimple element with 
$$C_{G_1}(s_1) = \prod^{d^*}_{d=3}\GL_{1}(q^d)^{\fps_d/2}.$$
If $s_1 \in \Omega_{2m+1}(q)$ then we choose $s := \diag(I_{2n-2m+1},s_1)$, and 
if $s_1 \notin \Omega_{2m+1}(q)$ then we choose 
$s:=\diag(I_{2n-2m-3},s_2,s_1)$ where $s_2$ is defined as in 3). As above, 
$s$ gives rise to a regular character $\rho$ of $G$ which is trivial at $Z(G)$,
so $\rho$ can be viewed as an irreducible character of $S := \PSp_{2n}(q)$. The 
same arguments as in 3) now show that (\ref{for-cs1}) holds with
$B = 71/72$, and (\ref{for-cs2}) holds with $Be^{-9/16} > 1/2$.

\smallskip
4b) Assume now that $G = \SO_{2n+1}(q)$ with $q$ odd. Then we choose 
$d^*$ and $m$ as in 4a), and choose $s \in G^* = \Sp_{2n}(q)$ a semisimple 
element with 
$$C_{G}(s) = \Sp_{2n-2m}(q) \times \prod^{d^*}_{d=3}\GL_{1}(q^d)^{\fps_d/2}.$$
Let $\chi$ be an irreducible constituent over $S := \Omega_{2n+1}(q)$ of the 
regular character labeled by $(s)$. The 
same arguments as in 3) now show that (\ref{for-cs1}) holds with
$B = (71/72) \cdot (1/2)$, and (\ref{for-cs2}) holds with $Be^{-9/16} > 1/4$.

\smallskip
4c) Next suppose that $G$ is of type $C_n$ over $\F_q$ with $q$ even; in particular,
$G^* \cong \Sp_{2n}(q)$. Choose $d^* \geq 5$ largest possible such that
$m := \sum^{d^*}_{d=5}d\fps_d/2 \leq n$, and so instead of (\ref{for-d2}) we now have 
\begin{equation}\label{for-d3}
  d^*+1 > \log_q((2n+3)(1-1/q)+q^4).  
\end{equation}
We can find a semisimple element $s \in G^*$ such that
$$C_{G^*}(s) = \Sp_{2n-2m}(q) \times \prod^{d^*}_{d=3}\GL_{1}(q^d)^{\fps_d/2}.$$
By Lemma \ref{sum}(i), $\prod^{\infty}_{i=1}(1-q^{-2i}) > 11/16$. 
Considering the regular character $\rho$ labeled by $(s)$, we now obtain 
\begin{equation}\label{for-cs3}
  c(S) \geq \frac{\rho(1)}{\St(1)} >
  \frac{B}{\prod^{d^*}_{j=5}(1-q^{-j})^{\fps_j/2}},
\end{equation}
with $B = (11/16)$. Applying (\ref{poly2}) and arguing as in 3), we arrive at 
\begin{equation}\label{for-cs4}
  c(S) > 
  \frac{B}{e^{25/32}}\left(\log_q((2n+3)(1-\frac{1}{q})+q^4)\right)^{3/8}
  > \frac{B}{e^{25/32}}\left(\log_{q}(n+17)\right)^{3/8},
\end{equation}
and so we are done, since $Be^{-25/32} > 0.3$. 

\smallskip
4d) Finally, let $G = \Omega^{\e}_{2n}(q)$ with $q$ even and $n \geq 4$; in 
particular, $G^* \cong G$. Now we choose $d^*$ and $m$ as in 4c), and fix a
semisimple element $s \in G^*$ with 
$$C_{G^*}(s) = \Omega^{\e}_{2n-2m}(q) \times \prod^{d^*}_{d=3}\GL_{1}(q^d)^{\fps_d/2}.$$
Using the estimate $(1-\e q^{-n})/(1-\e q^{m-n}) > 2/3$ and arguing as in 4c),
we see that (\ref{for-cs3}) holds with 
$B = (11/16) \cdot (2/3) = 11/24$. Consequently, (\ref{for-cs4}) holds with
$Be^{-25/32} > 1/5$. 
\end{proof}

Following the same approach, A. Schaeffer has proved:

\begin{thm}   \label{bound6} {\rm \cite{Sc}}
 Let $G$ be any of the following twisted Lie-type groups of type $A$: 
 $\GU_n(q)$, $\PGU_n(q)$, $\SU_n(q)$, or $\PSU_n(q)$. 
Then the following inequalities hold:
$$\max\left\{1, \frac{1}{4}
           \left(\log_{q}((n-1)(1-\frac{1}{q^2})+q^4)\right)^{2/5}\right\} 
  \leq \frac{b(G)}{q^{n(n-1)/2}} < 2(\log_q(n(q^2-1)+q^2))^{1.27}.$$
\end{thm}

\begin{proof}[Proof of Theorem \ref{main2}]
The cases where $G$ is an exceptional group of
Lie type follow from the proof of Proposition \ref{exc}. Consider the case 
$G$ is classical. Then the upper bound follows from Theorems \ref{bound4}, 
\ref{bound5}, and \ref{bound6}. We need only to add some explanation 
for the groups of type $A$, twisted  or untwisted. For instance, let $G$ be 
a group of Lie type in the same isogeny class with $L := \SL_n(q)$. Then
$G \cong (L/Z) \cdot C_d$, where $Z$ is a central subgroup of order $d$ of 
$L$, and furthermore the subgroup of all automorphisms of $L/Z$ induced by
conjugations by elements in $G$ is contained in $\PGL_n(q)$. Now consider
any $\chi \in \Irr(G)$. Let $\chi_1$ be an irreducible constituent of 
$\chi|_{L/Z}$ viewed as a characters of $L$ and let $\chi_2$ be an irreducible
constituent of $\Ind^{H}_{L}(\chi_1)$, where $H := \GL_n(q)$. Since the quotients
$G/(L/Z)$ and $H/L$ are cyclic, we see that $\chi(1)/\chi_1(1)$ is the index
(in $G$) of the inertia group of $\chi_1$ in $G$, which is at most the index 
(in $H$) of the inertia group of $\chi_1$ in $H$, and the latter index is
just $\chi_2(1)/\chi_1(1)$. It follows that $\chi(1) \leq \chi_2(1) \leq b(H)$.
The same argument applies to the twisted case of type $A$.

For the lower bound, observe that 
there is some $d_{\varep} \geq 5$ depending on $\varep$ such that
$$\fp_d \geq \fps_d > (1-\varep)\frac{q^d}{d}.$$ 
Choosing $A \leq (d_{\varep})^{(\varep-1)/\gamma}$  we can guarantee that the 
lower bound holds for $n \leq q^{d_{\varep}}$. Hence we may assume that 
$n \geq q^{d_{\varep}}+1 \geq d_{\varep}\fp_{d_{\varep}}+3$. Now we can repeat the proofs
of Theorems \ref{bound4} and \ref{bound5}, replacing the products
$\prod^{d^*}_{d=3}$, respectively $\prod^{d^*}_{d=5}$, by 
$\prod^{d^*}_{d = d_{\varep}}$.
\end{proof} 

\begin{proof}[Proof of Theorem \ref{main3}]
To guarantee the lower bound in the case $\ell = p$ we can take $C \leq 1$,
since the Steinberg, being of $p$-defect~0, is irreducible modulo $p$. Assume
that $\ell \neq p$. As usual, by choosing $C$ small enough we can ignore any
finite number of simple groups; also, it suffices to prove the lower bound
for the unique non-abelian composition factor $S$ of $G$. So we will work
with $S = G/Z(G)$, where $\cG$ is a simple simply connected algebraic group
and $G = \cG^F$ is the corresponding finite group over $\F_q$. Consider
the pair $(\cGD,\FD)$ dual to $(\cG,F)$ and the dual group $G^*:=(\cGD)^{\FD}$.
It is well known that, for $q \geq 5$, ${\mathrm {IBr}}_{\ell}(\PSL_2(q))$ 
contains a character of
degree $\geq q-1$, so we may assume that $r:=\rank(\cG) > 1$. We will show
that, with a finite number of exceptions, $[G^*,G^*]$ contains a regular semisimple 
$\ell'$-element $s$ with connected centralizer and such that $C_{G^*}(s)$ is a 
torus of order $\leq 2q^r$. For such an $s$, the corresponding semisimple character
$\chi = \chi_s$ can be viewed as an irreducible character of $S$ of degree
$|G|_{p'}/|C_{G^*}(s)| > C\cdot|G|_p$ (with $C > 0$ suitably chosen). 
Moreover, any Brauer 
character in the $\ell$-block of $G$ containing $\chi$ has degree divisible by 
$\chi(1)$ as a consequence of a result of Brou\'e--Michel, see
 \cite[Prop. 1]{HM}. Hence the reduction modulo $\ell$ of $\chi$ is 
irreducible and so $b_{\ell}(S) \geq \chi(1)$.    

To find such an $s$, we will work with two specific tori $T_1$ and $T_2$ of $G^*$.
For $G=\tw3D_4(q)$ we can choose $|T_1|=q^4-q^2+1$ and $|T_2|=(q^2-q+1)^2$.
For $G = \SU_n(q)$, we choose 
$$(|T_1|,|T_2|) = \begin{cases}
  \left(\frac{(q^{n/2}+1)^2}{q+1}, q^{n-1}+1\right)&
     \text{if $n \equiv 2\pmod4$},\cr 
  \left(\frac{q^n+1}{q+1},(q^{(n-1)/2}+1)^2\right)&
     \text{if $n \equiv 3\pmod4$}\end{cases}$$
If $G$ is of type $B_n$ or $C_n$ with $2|n$, we can choose
$$(|T_1|,|T_2|) = (q^n+1,(q^{n/2}+1)^2).$$
For all other $G$, $T_1$ and $T_2$ can be chosen of order indicated in 
\cite[Tables~3.5 and~4.2]{Ma2}. 
We may assume that either $q$ or the rank of $G$ is sufficiently large, so in 
particular Zsigmondy primes $r_i$ \cite{Zs} exist for the cyclotomic 
polynomials $\Phi_{m_i}(q)$ of largest possible $m_i$ dividing the orders $|T_i|$. 
Here $i = 1,2$, and, furthermore, for $i = 2$ we need to assume that $G$ is not 
$\SL_3(q)$, $\SU_3(q)$, or $\Sp_4(q) \cong \Spin_5(q)$. According to 
\cite{F}, either $r_2 > m_2+1$ or $r_2^2$ divides $\Phi_{m_2}(q)$, again with 
finitely many exceptions. 

Now if $r_1 \neq \ell$, respectively if $\ell = r_1 \neq r_2$ and 
$r_2$ is larger than all torsion primes of $\cG$
(see e.g. \cite[Table 2.3]{MT} for the list of them), we can choose 
$s \in T_i$ of prime order $r_i$, with $i = 1$, respectively $i=2$, and 
observe that $r_i$ is coprime to all torsion primes of $\cG$ 
as well as to $|G^*/[G^*,G^*]|$. It follows that $C_{\cGD}(s)$ is connected
(cf. \cite[Prop. 14.20]{MT} for instance), 
$s \in [G^*,G^*]$, and moreover $s$ can be chosen so that 
$|C_{G^*}(s)| = |T_i|$. Thus $s$ has the desired properties, and so we are done. 

We observe that $r_2$ can be a torsion prime for $\cG$ only when 
$r_2 = m_2+1$ and $(G,r_2)$ is $(\SL_n(q),n)$, or $(\SU_n(q),n)$ with
$n \equiv 3 \pmod4$. In either case we can choose $s \in T_2 \cap [G^*,G^*]$ of
order $r_2^2$. Furthermore, if $G = \SL^{\e}_3(q)$ with $q \geq 5$, we fix
$\al \in \F_{q^2}^{\times}$ of order $q+\e$ and choose $s \in T_2$ with an inverse
image $\diag(\al,\al^{-1},1)$ in $\cG$. If $G = \Sp_4(q)$ with $q \geq 8$, 
we fix $\beta \in \F_{q^2}^{\times}$ of order 
$q+1$ and choose $s \in T_2$ with an inverse image 
$\diag(\beta,\beta^{-1},\beta^2,\beta^{-2})$ in 
$\Sp_4(\overline{\F}_q) \cong \Spin_5(\overline{\F}_q)$. 
It remains to show that in these cases the element $s$ has the desired
properties. In fact,
it suffices to show that $C_{\cGD}(s)$ is a torus. Consider for instance the
case $G = \SL_n(q)$ (so $r_2 = n$). 
Then $s$ can be chosen to have an inverse image 
$\diag(\gam,\gam^q, \ldots,\gam^{q^{n-2}},1)$ in the 
simply connected group $\hat\cG^*$, where
$|\gam| = n^2$ and $\cGD = \hat\cG^*/Z(\hat\cG^*)$. Suppose $x \in \hat\cG^*$
centralizes $g$ modulo $Z(\hat\cG^*)$. Then 
$xgx^{-1} = \delta g$ for some $\delta \in \overline{\F}_q^{\times}$ with
$\delta^n = 1$. It follows that $\delta$ is an eigenvalue of $g$ of order
dividing $n$, and so $\delta = 1$. Thus $C_{\cGD}(s)$ equals
$C_{\hat\cG^*}(s)/Z(\hat\cG^*)$ and so it is a torus. 
Similar arguments apply to all the remaining cases.
\end{proof}


\end{document}